\documentclass[11pt]{amsart}  

\usepackage{amsmath}
\usepackage{amsfonts}
\usepackage{amssymb}
\usepackage{graphicx}
\usepackage{hyperref}
\usepackage{tikz}
\usepackage[margin=1in]{geometry}

\newcommand{\Span}{\operatorname{span}}
\newcommand{\f}{{\fI}}

\newcommand{\g}{{\gI}}

\newcommand{\fI}{{\{f_i\}_{i \in I}}}
\newcommand{\fJ}{{\{f_j\}_{j \in J}}}

\newcommand{\pfJ}{{\{\pi f_j\}_{j \in J}}}
\newcommand{\gI}{{\{g_i\}_{i \in I}}}

\newcommand{\diagf}{{\{\diag{f_i}\}_{i \in I}}}
\newcommand{\fdiag}{\diag{f}}

\newcommand{\gdiag}{\diag{g}}
\newcommand{\diagfI}{{\{\diag{f_i}\}_{i \in I}}}

\newcommand{\supp}{\operatorname{supp}}
\newcommand{\R}{\mathbb{R}}
\newcommand{\N}{\mathbb{N}}
\newcommand{\Z}{\mathbb{Z}}
\newcommand{\Q}{\mathbb{Q}}
\newcommand{\sub}{\subseteq}
\newcommand{\C}{\mathbb{C}}
\newcommand{\tG}{\widetilde{G}}%stands for tilde G
\newcommand{\conv}{\operatorname{conv}} % stands for conv, as in convex hull
\newcommand{\diag}[1]{\widetilde{#1}}

\newcommand{\Hn}{H_n}
\newcommand{\F}{F}
\newcommand{\G}{G}

\newcommand{\fpos}[1]{\mathbb{F}_{#1}}
\newcommand{\fpF}{\fpos{\F}}
\newcommand{\fpG}{\fpos{\G}}

\newcommand{\Hnwozero}{\Hn\setminus\{0\}}

\newcommand{\vect}[1]{\begin{bmatrix} #1 \end{bmatrix}}

\newcommand{\aaa}{\mathbf{a}}

\newcommand{\mb}[1]{\mathbf{#1}}
\newcommand{\floor}[1]{\left\lfloor #1 \right\rfloor}
\newcommand{\ceil}[1]{\left\lceil #1 \right\rceil}
\newcommand{\subsetsum}[1]{\operatorname{SubSum}\left(#1\right)}
\newcommand{\spandex}[1]{\mathcal{I}\left(#1\right)}

\newcommand{\Real}[0]{\mathbb{R}}

\newcommand{\inter}{\operatorname{relint}}

\newcommand{\Poly}{\mathcal{P}}
\newcount\colveccount
\newcommand*\colvec[1]{
        \global\colveccount#1
        \begin{bmatrix}
        \colvecnext
}
\def\colvecnext#1{
        #1
        \global\advance\colveccount-1
        \ifnum\colveccount>0
                \\
                \expandafter\colvecnext
        \else
                \end{bmatrix}
        \fi
}
\newcommand{\calD}{\mathcal{D}}

\newcommand{\ms}{\{v_i\}_{i=1}^m}

\newcommand{\inpro}[1]{\langle #1 \rangle}

\newcommand{\hb}[1]{\mathfrak{h}_{\{#1\}}}
\newcommand{\h}{\mathfrak{h}}
\newcommand{\set}[1]{\left\{#1\right\}}

\newcommand{\rowred}{\operatorname{rr}}

\newcommand{\Us}{\mathcal{S}} % unit sphere notation
\newcommand{\US}{\Us}

\theoremstyle{remark}
\newtheorem{theorem}{Theorem}
\newtheorem{lemma}{Lemma}
\newtheorem{corollary}{Corollary}
\newtheorem{proposition}{Proposition}
\newtheorem{remark}{Remark}
\newtheorem{example}{Example}
\newtheorem{definition}{Definition}
\newtheorem{conjecture}{Conjecture}
\newtheorem{obs}{Observation}

\title{On Structural Decompositions of Finite Frames}

\author[Chan et al.]{Alice Z.-Y. Chan}
\address{Department of Mathematics, UC-San Diego}
\email{azchan[at]ucsd[dot]edu}

\author[]{Martin S. Copenhaver}
\address{Operations Research Center, MIT}
\email{mcopen[at]mit[dot]edu}

\author[]{Sivaram K. Narayan}
\address{Department of Mathematics, Central Michigan University}
\email{naray1sk[at]cmich[dot]edu}

\author[]{Logan Stokols}
\address{Department of Mathematics, UT Austin}
\email{lstokols[at]math[dot]utexas[dot]edu}

\author[]{Allison Theobold}
\address{Department of Statistics, Montana State University}
\email{allison.theobold[at]msu[dot]montana[dot]edu}

\thanks{Research supported  by NSF-REU Grant DMS 11-56890.}


\date{\today}


\begin{document}

\begin{abstract}
A \emph{frame} in an $n$-dimensional Hilbert space $H_n$ is a possibly redundant collection of vectors $\{f_i\}_{i\in I}$ that span the space. A \emph{tight} frame is a generalization of an orthonormal basis. A frame $\{f_i\}_{i\in I}$ is said to be \emph{scalable} if there exist nonnegative scalars $\{c_i\}_{i\in I}$ such that $\{c_if_i\}_{i\in I}$ is a tight frame. In this paper we study the combinatorial structure of frames and their decomposition into tight or scalable subsets by using partially-ordered sets (posets). We define  the \emph{factor poset} of a frame $\{f_i\}_{i\in I}$ to be a collection of subsets of $I$ ordered by inclusion so that nonempty $J\subseteq I$ is in the factor poset if and only if $\{f_j\}_{j\in J}$ is a tight frame for $H_n$. A similar definition is given for the scalability poset of a frame. We prove conditions which factor posets satisfy and use these to study the \emph{inverse factor poset problem}, which inquires when there exists a frame whose factor poset is some given poset $P$. We determine a necessary condition for solving the inverse factor poset problem in $H_n$ which is also sufficient for $H_2$. We describe how factor poset structure of frames is preserved under orthogonal projections. We also consider the enumeration of the number of possible factor posets and bounds on the size of factors posets. We then turn our attention to scalable frames and present partial results regarding when a frame can be scaled to have a given factor poset.
\end{abstract}

\maketitle

%%%%%%%%%%%%% Introduction

\section{Introduction}

Frames are systems for representing finite- or infinite-dimensional Hilbert space elements. A frame for a finite-dimensional Hilbert space is a spanning set that is not necessarily a basis. The concept of frames was introduced by Duffin and Schaeffer \cite{duffin} in 1952. Frames began to be studied widely after the landmark paper of Daubechies, Grossman and Meyer \cite{daub} in 1986. They have become a crucial component in the state-of-the-art techniques in signal processing, information theory, engineering, and computer science. The areas of wavelets and, more recently, compressive sensing are just two settings in which frames play an essential role.

Redundancy of frame vectors in finite-dimensional frames plays a pivotal role in the construction of stable signal representations and in mitigating the effect of losses in transmission of signals through communication channels. Because of the usefulness in applications, finite frames have been studied intensively in the recent years \cite{finiteframes}. 

In this paper, we study two of the important classes of frames in finite-dimensional Hilbert spaces $H_n$: tight frames and scalable frames. A tight frame $F$ for $H_n$ is said to be $\emph{prime}$ if no proper subset of $F$ is a tight frame for $\Hn$. Therefore, every tight frame can be written as a union of prime tight frames \cite{prime}. This leads us to study the combinatorial structure of tight frames and scalable frames. In Section 2 we present the necessary preliminaries from finite frame theory. Factor posets are studied in detail in Section 3. Some necessary conditions are given for a poset to be a factor poset of a frame. We then answer the question in $H_2$ of finding a frame whose factor poset equals a given poset, known as the \emph{inverse factor poset problem}. In addition we study how factor poset structure is preserved under orthogonal projections. We also present some bounds on the size of factor posets and consider the problem of enumerating all possible factor posets with a given number of vectors. In Section 4 we turn our attention to scalable frames and the scalability poset. We characterize scalings that produce non-prime frames and close with an open problem.

\section{Preliminaries}

Throughout this paper we only consider vectors in $\R^n$ or $\C^n$. If a result holds true in both $\R^n$ and $\C^n$ then we will indicate that the result holds for any $n$-dimensional Hilbert space $H_n$. Otherwise, we will state the result only for $\R^n$. A frame is defined as follows.

\begin{definition} A sequence $\fI\sub\Hn$ is a frame for $\Hn$ with frame bounds $0<A\leq B<\infty$ if for all $x\in\Hn$,
\begin{equation}\label{frameDefn}
A\|x\|^2 \leq \sum_{i\in I}|\langle x, f_i\rangle|^2 \leq B\|x\|^2.
\end{equation}
\end{definition}

It has been proven that in $\Hn$ a frame is equivalent to a spanning set.

\begin{theorem}[{{\cite[p. 99]{ffu}}}] A sequence of vectors in $\Hn$ is a frame for $\Hn$ if and only if it is a spanning set for $\Hn$.
\end{theorem}

A frame can be viewed as a generalization of a basis. There are special types of frames that generalize an orthonormal basis. 

\begin{definition} A frame $\f$ is said to be $\lambda$-\emph{tight} if $\lambda =A=B$ in Eq. \eqref{frameDefn} and is said to be $\emph{Parseval}$ if $A=B=1$. 
\end{definition}

The following operators associated with frames are useful in the study of frames. Henceforth, for convenience, we fix the set of indices of a frame as $I = \{1,2,\ldots,k\}$. For a sequence $\fI$ we define the $\emph{analysis}$ operator $\theta : \Hn \rightarrow H_k$ by
$$\theta(x) = \sum_{i=1}^k \langle x,f_i\rangle e_i,$$
where $\{e_i\}_{i\in I}$ is an orthonormal basis for $H_k$. The adjoint of $\theta$, namely $\theta^*: H_k \rightarrow \Hn$, is defined by 
$$\theta^*(e_i) = f_i$$ for $i \in I$.
This operator $\theta^*$ is called the $\emph{synthesis}$ operator. Using the standard orthonormal bases for the spaces $H_n$ and $H_k$ the operators $\theta$ and $\theta^*$ can be represented by matrices as 
$$\theta^* = \begin{bmatrix} \uparrow & & \uparrow \\ f_1 & \cdots & f_k \\ \downarrow & & \downarrow \end{bmatrix}\text{\quad and \quad} \theta = \begin{bmatrix} \leftarrow f_1^* \rightarrow \\ \vdots \\ \leftarrow f_k^* \rightarrow 
\end{bmatrix}.$$ 

The $\emph{frame}$ operator $S : \Hn \rightarrow \Hn$ is given by $S (x) = \theta^* \theta( x )= \theta^* (\sum_{i\in I} \langle x,f_i \rangle e_i) = \sum_{i\in I} \langle x, f_i \rangle f_i$. The $\emph{Gramian}$ operator $G:H_k \rightarrow H_k$ is given by $G=\theta\theta^*$ and is represented by the $k\times k$ matrix $[G_{ij}]=[\langle f_j, f_i \rangle]$. The following theorem gives some important equivalent formulations of frames.

\begin{theorem} [{{\cite[pp. 105-107]{ffu}}}]
Let $\fI$ be a sequence of vectors in $\Hn$ with associated operators defined as above. The following hold:
\begin{enumerate}
\item $\fI$ is a frame if and only if $S$ is an invertible (positive) operator.
\item $\fI$ is a $\lambda$-tight frame if and only if $S=\lambda I_n$, where $I_n$ denotes the identity matrix of size $n$. 
\item $\fI$ is a Parseval frame if and only if $S=I_n$. 
\end{enumerate}
\end{theorem}

\subsection{Diagram Vectors}
Any vector $f$ in $\R^2$ can be written as $f = \begin{bmatrix} f(1) \\ f(2) \end{bmatrix}$. The associated diagram vector $\fdiag$ is defined as $\fdiag = \begin{bmatrix} f^2(1)-f^2(2) \\ 2f(1)f(2) \end{bmatrix}$. Diagram vectors are used in the following characterization of tight frames.
\begin{theorem} [{{\cite[p. 124]{ffu}}}]\label{thm:diagSumR2}
Let $\fI$ be a sequence of vectors in $\R^2$, not all of which are zero. Then $\fI$ is a tight frame if and only if $\sum_{i\in I} \diag{f}_i = 0$. 
\end{theorem} 

The diagram vector of a vector in $\R^2$ belongs to $\R^2$. The diagram vectors of a tight frame in $\R^2$ can be placed tip-to-tail to demonstrate that they sum to zero. In \cite{tfs} the notion of diagram vectors was extended to $\R^n$ and $\C^n$. 

\begin{definition}
For any vector $f = \begin{bmatrix} f(1) \\ \vdots \\ f(n) \end{bmatrix}\in\R^n$, we define the diagram vector of $f$, denoted as $\fdiag$, by 
$$\fdiag = \frac{1}{\sqrt{n-1}} \begin{bmatrix} f^2(1)-f^2(2) \\ \vdots \\ f^2(n-1)-f^2(n) \\ \sqrt{2n}f(1)f(2) \\ \vdots \\ \sqrt{2n}f(n-1)f(n) \end{bmatrix} \in \R^{n(n-1)},$$
where the difference of squares $f^2(i)-f^2(j)$ and the product $f(i)f(j)$ occur exactly once for $i<j, i=1,\ldots,n-1$. 
\end{definition}

\begin{definition} 
For any vector $f \in \C^n$, we define the diagram vector of $f$, denoted as $\fdiag$, by
$$\fdiag = \frac{1}{\sqrt{n-1}} \begin{bmatrix} |f(1)|^2-|f(2)|^2 \\ \vdots\\|f(n-1)|^2-|f(n)|^2 \\ \sqrt{n} f(1)\overline{f(2)} \\ \sqrt{n} \overline{f(1)}f(2) \\ \vdots \\ \sqrt{n}  f(n-1)\overline{f(n)} \\ \sqrt{n}  \overline{f(n-1)}f(n) \end{bmatrix} \in \C^{3n(n-1)/2},$$
where the difference of the form $|f(i)|^2-|f(j)|^2$  occurs exactly once for $i<j, i=1,\ldots,n-1$, and the product of the form $f(i)\overline{f(j)}$ occurs exactly once for $i\neq j$. 

\end{definition}

Using these definitions, Theorem \ref{thm:diagSumR2} was extended to $H_n$ in \cite{tfs}.

\begin{theorem} [\cite{tfs}]
Let $\fI$ be a sequence of vectors in $H_n$, not all of which are zero. Then $\fI$ is a tight frame if and only if $\sum_{i \in I} \tilde{f}_i = 0$. Moreover, for any $f,g\in\Hn$ we have ${(n-1)}\langle \fdiag, \gdiag \rangle = n |\langle f, g \rangle|^2 - \|f\|^2 \|g\|^2$. 
\end{theorem} 

\begin{remark} [\cite{tfs}]
From the above theorem it is immediate that if $\|f\| = 1$ then $\|\fdiag\| = 1$. Suppose $\Us^k := \{f \in \R^{k+1} : \|f\| = 1\}$ is the unit sphere in $\R^{k+1}$. We can define the diagram operator $D : S^{n-1} \rightarrow S^{n(n-1)-1}$ and note that $D$ is not injective since $f$ and $-f$ have the same diagram vectors. It can be shown that $D$ is surjective if and only if $n = 2$. 
\end{remark}

\subsection{Posets}
One of the main tools we use for studying the combinatorial structure of finite frames is partially ordered sets (or posets). We recall the definition of a poset. Let $P$ be a nonempty set. A $\emph{partial order}$ on $P$ is a relation denoted by $\leq$ satisfying the following properties:

\begin{enumerate}
\item $a \leq a$ for every $a \in P$ (reflexivity)
\item $a \leq b$ and $b \leq a$ implies $a = b$ (anti-symmetry)
\item $a \leq b$ and $b \leq c$ implies $a \leq c$ (transitivity) 
\end{enumerate} 

We may write $a \leq b$ in the equivalent form $b \geq a$. A nonempty set $P$ with a partial order is called a $\emph{partially ordered set}$ or $\emph{poset}$.

Two elements $a$ and $b$ in a poset $P$ are said to be $\emph{comparable}$ if one of them is less than or equal to the other, that is, if $a \leq b$ or $b \leq a$. A partially ordered subset in which any two elements are comparable is called a $\emph{totally ordered set}$ or $\emph{chain}$. A partially ordered subset in which no two elements are comparable is called an $\emph{antichain}$ or a $\emph{Sperner family}$.

A poset can be represented by a Hasse diagram. If the partial order on the power set of $\{1,2,3,4\}$ is given by inclusion, then the poset 
$\{\{1,2,3,4\},\{1,2\},\{3,4\},\emptyset\}$ is represented by the following Hasse diagram:
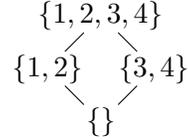
\begin{center}
\begin{tikzpicture}
    \node (top) at (0,0) {$\{1,2,3,4\}$};
    \node [below left  of=top] (left) {$\{1,2\}$};
 \node [below right of=top] (right) {$\{3,4\}$};
 \node [below right of =left] (bot) {$\{\}$};
 \draw [ shorten <=-4pt,shorten >=-4pt] (top) -- (left);
  \draw [ shorten <=-4pt,shorten >=-4pt] (top) -- (right);
  \draw [ shorten <=-4pt,shorten >=-4pt] (bot) -- (left);
  \draw [ shorten <=-4pt,shorten >=-4pt] (bot) -- (right);
\end{tikzpicture}
\end{center}

\noindent All of the posets which we will consider in this paper will be collections of sets, ordered by set inclusion.

\subsection{Prime Tight Frames}
An important concept in the study of tight frames is the notion of prime tight frames defined and studied in \cite{prime}. 

\begin{definition}[\cite{prime}]
A tight frame $\fI\sub H_n$ is said to be $\emph{prime}$ if no proper subset of the frame is a tight frame for $H_n$. If $\fI$ is not prime, then we say the frame is $\emph{divisible}$.
\end{definition}

One of the main results in \cite{prime} is the following.

\begin{theorem} [\cite{prime}]
Every tight frame of $k$ vectors in $H_n$ is a finite union of prime tight frames. 
\end{theorem}

\begin{definition}
Let $F$ be a tight frame. Then for some $\ell \in \N$, $$F = F_1 \cup F_2 \cup \cdots \cup F_{\ell},$$ where $F_j$ is a prime tight frame for $j = 1,2,\ldots,\ell$. We say that the $F_j$ are $\emph{prime factors}$ or $\emph{prime divisors}$ of $F$. 
\end{definition}

%%%%%%%%%%%% Factor posets

\section{Factor posets}

In this section we study the decompositions of frames, particularly tight frames, into tight subframes. We begin by defining a poset structure which describes such a decomposition. \emph{For simplicity in the remainder of this section we only consider frames which contain no zero vectors.}

\begin{definition}
Let $F=\fI\sub H_n$ be a frame. We define its factor poset $\fpF\sub 2^I$ to be the set
$$\fpF = \{J\sub I:\fJ \text{ is a tight frame}\}$$
partially ordered by set inclusion. We assume $\emptyset\in \fpF$.
\end{definition}

\begin{definition}
For a frame $F=\fI\sub H_n$ and its factor poset $\fpF$, we define the \emph{empty cover of $\fpF$}, $EC(\fpF)$, to be the set of $J\in \fpF$ which cover $\emptyset$, that is,
$$EC(\fpF) = \{J\in \fpF: J\neq \emptyset \text{ and} \not\exists J'\in \fpF \text{ with }\emptyset \subsetneq J'\subsetneq J\}.$$
\end{definition}

\begin{example}
Let $F=\{e_1,e_2,-e_1,-e_2\}=\{f_1,f_2,f_3,f_4\}\sub \R^2$. Then
$$\fpF = \{\emptyset,\{1,2\},\{2,3\},\{3,4\},\{4,1\},\{1,2,3,4\}\}.$$
This follows from finding the diagram vectors:
$$\diag{f}_1 = \diag{f}_3 = e_1\text{ and }\diag{f}_2=\diag{f}_4 = -e_1.$$
Here we note the correspondence between tight subframes and collections of diagram vectors which sum to zero. Because we consider frames without zero vectors, this is a one-to-one correspondence. Here we also compute $EC(\fpF) =\{\{1,2\},\{2,3\},\{3,4\},\{4,1\}\}$.
\end{example}

\subsection{Inverse Factor Poset Problem}

In this section we consider the \emph{inverse factor poset problem}, namely, the problem of determining for which posets $P\sub 2^I$ there exists a frame in $H_n$ whose factor poset is $P$. Given the correspondence between tight subframes and collections of diagram vectors summing to zero, we begin by considering the well-studied \emph{subset-sum problem}, which for a finite set $A\sub\R$ inquires whether there exists a subset whose sum is zero (or some fixed number). This problem has implications in computer science, number theory, and beyond, and is a classical example of an \textsf{NP}-complete problem \cite{algs}. 

In the context of the subset-sum problem for a set $A$, a natural poset to consider is
$$\subsetsum{A}:=\left\{B\sub A: \sum_{b\in B} b = 0\right\}.$$
In the following theorem we prove that solving the inverse factor poset problem for frames in $\R^2$ is equivalent to determining whether a poset corresponds to $\subsetsum{A}$ for some set $A\sub\R$. We begin by stating a result from \cite{2012report}. We include a proof here for convenience of exposition.

\begin{lemma}[{{\cite[Theorem 3.1]{2012report}}}]\label{lemma:scaledONBnEquals2}
Suppose $\F = \f\sub\R^2$ is a collection of vectors. Then there exists a corresponding collection of vectors $\G=\g\sub\R^2$ with $g_i \in\Span\{v_1\}\cup \Span\{v_2\}$ for all $i\in I$ so that $\fpos{\F}=\fpos{\G}$, where $\{v_1,v_2\}$ is any orthonormal basis for $\R^2$.
\end{lemma}

\begin{proof}
By the invariance of factor poset structure under rotations we take without loss of generality $v_1=e_1$ and $v_2=e_2$, where $e_1=(1,0)$ and $e_2=(0,1)$. We now define a linear transformation $T:\R^2\to\R$ such that the following property holds:
$$(\dagger)\quad\text{ For all }J\sub I, \;T\left(\sum_{j\in J}\diag{f}_j\right)=0 \text{ if and only if }\sum_{j\in J}\diag{f}_j=0 .$$
Because these are linear transformations, it is true that $T\left(\sum_{j\in J}\diag{f}_j\right)=0$ if $\sum_{j\in J}\diag{f}_j = 0$. We now seek such a $T$ so that the converse holds. Note that all linear transformations $T$ can be written in the form $T(x,y) = \alpha x +\beta y$. Let $A = \{J\sub  I: \sum_{j\in J} \diag{f}_j\neq 0\}$. Consider
$$R = \bigcup_{J\in A} \left\{ (\alpha,\beta)\in\R^2: \alpha\left(\sum_{j\in J}\diag{f}_j(1)\right)+ \beta\left(\sum_{j\in J}\diag{f}_j(2)\right)=0\right\}.$$
Note that $R$ is a finite union of one-dimensional subspaces, and hence $R^c\sub\R^2$ is non-empty. Choose $(\alpha,\beta)\in R^c$. The corresponding linear transformation $T:\R^2\to\R$ with
$$T(x,y) = \alpha x+\beta y$$
then satisfies the desired property $(\dagger)$. Let 
\begin{align*}
S_p &= \{i\in I: T(\diag{f_i})>0\}\\
S_z &= \{i\in I: T(\diag{f_i})=0\}\\
S_n &= \{i\in I: T(\diag{f_i})<0\}.
\end{align*}
For each $i\in I$, we define $g_i$ as follows:
\begin{enumerate}
\item If $i\in S_p$, let $g_i = \sqrt{T(\diag{f_i})}e_1$. Then $\diag{g_i} = (T(\diag{f_i}),0)$.
\item If $i\in S_z$, let $g_i = 0$. Then $\diag{g_i} = (0,0)$. 
\item If $i\in S_n$, let $g_i = \sqrt{|T(\diag{f_i})|}e_2$. Then $\diag{g_i} = ( T(\diag{f_i}),0)$.
\end{enumerate}
Then for $J\sub I$ we have that $\sum_{j\in J}\diag{g_j} = 0$ if and only if $\sum_{j\in J} T(\diag{f_j})= T\left(\sum_{j\in J} \diag{f_j}\right)=0$ if and only if $\sum_{j\in J}\diag{f_j} = 0$. It follows that $\fpos{\F}=\fpos{\G}$ for $\G=\g$.
\end{proof}

\begin{proposition}\label{prop:subsetsumConnectionR2}
Given any collection $\F = \f\sub\R^2$ there exists a set $A\sub\R$ with $|A|=|I|$ so that $\subsetsum{A}=\fpF$. Conversely, given a set $A\sub\R$ there exists a collection of vectors $\F=\f\sub\R^2$ with $|A|=|I|$ so that $\subsetsum{A}=\fpF$. Further, this correspondence can be taken so that nonzero vectors in $\R^2$ correspond to nonzero numbers in $\R$, and vice versa.
\end{proposition}

\begin{proof}
To prove the general claim, we may assume $\F = \f\sub\R^2\setminus\{0\}$. By Lemma \ref{lemma:scaledONBnEquals2} we may map the diagram vectors $\diagfI$ to real numbers $\{a_i\}_{i\in I}\sub\R\setminus\{0\}$ so that for all $J\sub I$ we have
$$\sum_{j\in J}\diag{f}_j = 0 \text{ if and only if }\sum_{j\in J} a_j = 0.$$
In the reverse direction, given a set $A=\{a_i\}_{i\in I}\sub\R\setminus\{0\}$ we can define $\{f_i\}_{i\in I}\sub\R^2\setminus\{0\}$ in the same inverse procedure as given in the proof of Lemma \ref{lemma:scaledONBnEquals2}. This completes the proof.
\end{proof}

\begin{remark}\label{rmk:complexC2case}
Note that the analogue of Proposition \ref{prop:subsetsumConnectionR2} holds for frames in $\C^2$ as well, now relating them to subset-sums in $\C$. Observe that subset-sum posets in $\C$ are the intersection of two subset-sum posets for real numbers (namely, taking the subset-sum poset for the real parts and intersecting it with the subset-sum poset for the imaginary parts), and that the intersection of two such posets must itself be a subset-sum structure for real numbers by the following results in Theorem \ref{thm:inverse}. Therefore, we restrict our attention for now to factor posets for frames in $\R^n$.
\end{remark}

\begin{remark}
An analogue of Lemma \ref{lemma:scaledONBnEquals2} does not exist for $n>2$. For example, when $n=3$, consider the set
$$F = \left\{e_1,e_2,e_3,\frac{1}{\sqrt{2}}(e_2+e_3),\frac{1}{\sqrt{2}}(e_2-e_3)\right\}\sub\R^3,$$
where $\{e_1,e_2,e_3\}$ is an orthonormal basis for $\R^3$. This has factor poset 
$$\fpF =\{\emptyset,\{1,2,3\},\{1,4,5\}\}.$$
We claim there is no frame consisting of multiples of some orthonormal basis in $\R^3$ whose factor poset is $\fpF$. Suppose $G=\{g_1,g_2,g_3,g_4,g_5\}$ consists entirely of scaled copies of $e_1,e_2,e_3\in \R^3$ and that $\fpos{G}=\fpF$. Without loss of generality we may take $g_1 = e_1,g_2=e_2,g_3=e_3$ since $\{g_1,g_2,g_3\}$ must be tight in $\R^3$ (the only tight frames with $3$ vectors in $\R^3$ are orthogonal bases). Since $\{g_1,g_4,g_5\}$ must be tight as well, at least one of $g_4$ and $g_5$ must be $\pm e_2$ with the other being $\pm e_3$. Taking $g_4 = e_2$ and $g_5 = e_3$ would force that we also have $\{1,3,4\}\in \fpos{G}=\fpF$, which is a contradiction. Therefore, no such frame $G$ exists. This example has an obvious extension to $\R^n$ for $n>3$.
\end{remark}

We are now prepared to fully answer the inverse factor poset problem for frames in $H_2$. This will follow from a general condition which is \emph{necessary} for a solution to the inverse problem for frames in $H_n$. In the proof we shall use the following notation and terminology.

\begin{definition}
Fix $I=\{1,2,\ldots,k\}$ and the standard (real) orthonormal basis $\{e_i\}_{i=1}^k$ for $H_k$. For $J\sub I$ define the \emph{index vector for $J$} as
$$[J] := \sum_{j\in J} e_j.$$
Given $P\sub 2^I$, define the \emph{index span of $P$}, denoted $\spandex{P}$, as
$$\spandex{P} = \Span\{[J]:J\in P\}.$$
Here we take the span over real numbers regardless of whether $H_k$ is $\R^k$ or $\C^k$.
\end{definition} 

For example, taking $I=\{1,2,3,4\}$ and $P=\{\emptyset,\{1,2\},\{3,4\}\}$ we have in the case $H_4=\C^4$ that
$$[\{1,2\}] = e_1+e_2=\vect{1\\1\\0\\0}\in\C^4\text{ and }\spandex{P} = \Span\{0,e_1+e_2,e_3+e_4\}\sub\R^4.$$
We can now state a necessary condition for the inverse factor poset problem using the index span. Given that we restrict our attention to factor posets for frames with no zero vectors (as zero vectors interfere with the subset-sum structure of diagram vectors), we consider posets $P\sub 2^I$ which contain no singletons (as singletons in a factor poset would correspond to a diagram vector of zero).

\begin{theorem}\label{thm:inverse}
Let $I=\{1,2,\ldots,k\}$ be some finite index set and $P\sub 2^I$ be a poset ordered by set inclusion and which contains no singletons. If there exists some frame $F=\fI\sub\Hnwozero$ with factor poset $P$ then $P$ is \emph{span-closed}, i.e., $[J]\notin \spandex{P}$ for every $J\in 2^I\setminus P$.
\end{theorem}

\begin{proof}
Suppose there exists some frame $F=\fI\sub\Hnwozero$ with $\fpF = P$. First note that $\emptyset\in P$ since any factor poset contains the empty set as an element by definition. We may assume that $P$ contains an element other than $\emptyset$, and hence $\ell = \dim(\spandex{P})>0$. To  show that $P$ is span-closed, suppose $J\in 2^I$ with $[J]\in\spandex{P}$. We must prove that $J\in P$, i.e., $\sum_{j\in J} \diag{f}_j = 0$.  Since $[J]\in\spandex{P}$ we may write
$$[J] = \sum_{i=1}^\ell \alpha_i [J_i],$$
where $\alpha_i\in\R$, and $J_i\in P$ for all $i=1,\ldots,\ell$. Let $ \{\eta_i\}_{i=1}^{M}$ be an orthonormal basis for $\R^{n(n-1)}$ if $H_n=\R^n$ or an orthonormal basis for $\C^{3n(n-1)/2}$ if $H_n=\C^n$. To show $\sum_{j\in J} \diag{f}_j = 0$ it is enough to show that $\left\langle\sum_{j\in J}  \diag{f}_j,\eta_m\right\rangle = 0$ for all $m=1,\ldots,M$. For any such $m$ we have
\begin{align*}
\sum_{j\in J} \langle \diag{f}_j,\eta_m \rangle &= \left\langle \vect{\langle \diag{f}_1,\eta_m\rangle\\\vdots\\\langle \diag{f}_k,\eta_m \rangle },[J]\right\rangle\\
&=\left\langle \vect{\langle \diag{f}_1,\eta_m\rangle\\\vdots\\\langle \diag{f}_k,\eta_m \rangle },\sum_{i=1}^\ell \alpha_i [J_i]\right\rangle\\
&= \sum_{i=1}^\ell \alpha_i\left\langle \vect{\langle \diag{f}_1,\eta_m\rangle\\\vdots\\\langle \diag{f}_k,\eta_m \rangle }, [J_i]\right\rangle\\
&= \sum_{i=1}^\ell \sum_{j\in J_i} \langle \diag{f}_j,\eta_m\rangle\\
&= \sum_{i=1}^\ell 0 \quad\text{(since }J_i\in P)\\
&= 0.
\end{align*}
Hence $J\in P$, as was to be shown.

\end{proof}

We now prove that in $H_2$ a poset $P$ being span-closed is sufficient for the existence of a frame $F$ with factor poset $P$, whereas it is not sufficient for $H_n$ with $n>2$.

\begin{theorem}\label{thm:r2inverse}
Let $I=\{1,2,\ldots,k\}$ be some finite index set and $P\sub 2^I$ be a poset ordered by set inclusion and which contains no singletons. Then there exists some frame $F=\fI\sub H_2\setminus\{0\}$ with factor poset $P$ if and only if $P$ is \emph{span-closed}.
\end{theorem}

\begin{proof}
The forward direction is the content of Theorem \ref{thm:inverse}. We now prove with the reverse direction. Here we restrict our attention to $H_2=\R^2$ because the case for $H_2=\C^2$ is an obvious extension given Remark \ref{rmk:complexC2case}.

By Proposition \ref{prop:subsetsumConnectionR2} it suffices to show the existence of a vector $\mathbf{a}\in\R^k$ with no zero components so that
$$\langle \mathbf{a},[J]\rangle = 0\text{ if and only if } J\in P.$$
Note that $K=\spandex{P}$ has $\dim(K)<k$ because $P$ contains no singletons and hence $e_i\notin K$ for $i=1,\ldots,k$. Therefore, $K^\perp$ is a (linear) subspace of positive dimension. Write every vector $[J]$ for $J\in 2^I\setminus P$ uniquely as
$$[J] =[J]_{\|}+[J]_{\perp},$$
where $[J]_{\|}\in K$ and $[J]_\perp\in K^\perp$. We seek some $\aaa\in K^\perp$ so that $\langle \aaa,[J]\rangle \neq 0$ for all $J\in 2^I\setminus P$. For any $\aaa\in K^\perp$ and $J\in2^I\setminus P$ we have
$$\langle \aaa,[J]\rangle = \langle \aaa,[J]_\perp\rangle.$$
Note that the subspaces of the form $\{\aaa\in K^\perp:\langle \aaa,[J]_\perp\rangle=0\}$ are of codimension 1 in the ambient space $K^\perp$. Hence, the set
$$K^\perp\setminus \left(\bigcup_{J\in 2^I\setminus P}\{\aaa\in K^\perp:\langle \aaa,[J]\rangle=0\}\right)\neq\emptyset.$$
Hence, such an $\aaa =(a_1,\ldots,a_k)$ exists with $\sum_{j\in J} a_j= 0$ if and only if $J\in P$. Using the subset-sum reformulation of factor poset structure for frames in $\R^2$ from Proposition \ref{prop:subsetsumConnectionR2}, the proof is complete.
\end{proof}

\begin{example}
Let $P=\{\emptyset,\{1,2\},\{1,3\},\{1,4\},\{2,3,4\}\}\sub 2^{\{1,2,3,4\}}$. Note that $P$ cannot be the factor poset of any frame in $\R^2$ because $\{1,2\},\{1,3\},\{1,4\}$ imply that corresponding diagram vectors satisfy $\diag{f}_2=\diag{f}_3=\diag{f}_4$. However, one can readily verify that the conditions of Theorem \ref{thm:r2inverse} fail because $\dim(\spandex{P}) = 4 \not< k=4$, and hence $P$ could be span-closed only if $P = 2^{\{1,2,3,4\}}$.
\end{example}

Observe that Theorem \ref{thm:r2inverse} also determines for a given poset $P\sub 2^I$ the smallest possible factor poset which contains $P$. Namely, take $P$ and append any subsets $J\sub I$ for which $[J]\in \spandex{P}$. Call the new poset $P'$. This will necessarily make $P'$ span-closed. Note, however, that $P'$ may now contain singletons. As singletons correspond to zero vectors, $P'$ will not technically be a factor poset and is only the subset-sum poset for diagram vectors (which is not the same as the factor poset for the same frame when zero vectors are present), but the correspondence to an actual factor poset in this case is clear. In this sense the previous result determines precisely the deficiencies of a poset.

Let us also remark that Theorem \ref{thm:inverse} is not sufficient for $n>2$. For example, the factor poset $\{\emptyset,\{1,2\}\}$ for $\{e_1,e_2\}\sub\R^2$ certainly has no $\R^3$ inverse. Simple examples for $\R^n$ which do not violate such an obvious condition that every non-empty element of $P$ has size at least $n$ exist. Here is one such example:

\begin{example}
Let $P$ be the factor poset corresponding to the frame
$$F = \{f_i\}_{i=1}^9=\{e_1,e_1,e_1,e_1,e_1,e_1,\sqrt{2}e_2,\sqrt{2}e_2,\sqrt{2}e_2\}\sub\R^2,$$
which is a tight frame. Suppose we wish to construct a frame $\{g_i\}_{i=1}^9\sub\R^3$ with $P$ as its factor poset. Let us observe that
$$\{1,2,7\},\{1,3,7\},\{2,3,7\}\in P.$$
Because the only tight frames with $3$ vectors in $\R^3$ are orthogonal bases, we may assume without loss of generality that $g_1 =e_1$, $g_2=e_2$, and $g_7 = e_3$. Since $\{1,3,7\}\in P$, we have that $g_3 = \pm g_2$. Hence, $\dim(\Span\{g_2,g_3,g_7\}) = 2$, hence $\{2,3,7\}\notin P$, a contradiction. Therefore no frame in $\R^3$ has $P$ as its factor poset.
\end{example}

\begin{obs}\label{obs:rowsGiveDiagVecs}
We now revisit the inverse construction from subset-sum structures for real numbers to frames in $\R^2$ given in Proposition \ref{prop:subsetsumConnectionR2}. The construction involves turning a vector $\mb{v} =(v_1,\ldots,v_k)\in\Real^k$ into a collection of diagram vectors $$\colvec{2}{v_1}{0} ,\cdots, \colvec{2}{v_k}{0}.$$ 
As long as $\mb{v}\in K^\perp$, where $K=\spandex{P}$, the poset for this collection will contain the original poset $P$.  Moreover, as long as $\mb{v}$ does not lie in a \emph{forbidden hyperplane} $\h_J :=\{\aaa\in K^\perp:\langle \aaa,[J]\rangle=0\}$ for any $J \notin P$, the poset for this collection will not contain any subsets not in $P$.  Thus by choosing $\mb{v} \in K^\perp \setminus \bigcup_{J \notin P} \h_J$, we have a frame for the given poset. Because the forbidden hyperplanes have measure $0$ relative to $K^\perp$, this can be computed by picking any generic vector from $K^\perp$.

This construction can be extended to a more general construction.  If $\mb{v},\mb{w}\in\Real^k$, then we can turn them into a collection of $k$ diagram vectors $$\colvec{2}{v_1}{w_1}, \cdots, \colvec{2}{v_k}{w_k}.$$ 
We inquire about what conditions on $\mb{v}$ and $\mb{w}$ make this collection of diagram vectors correspond to a frame with factor poset $P$.  For any element in $P$, the sum of the corresponding diagram vectors should be $0$.  However, because the components of a vector sum ``independently,'' this simply means that for any element in $P$, the sum of the corresponding coordinates in $\mb{v}$ should be 0, and the same for $\mb{w}$.  Thus the poset for our collection of diagram vectors contains $P$ if and only if $\mb{v},\mb{w}\in K^\perp$.  

Conversely, we must inquire what conditions on $\mb{v}$ and $\mb{w}$ makes this collection's factor poset not have any subsets not in $P$.  For any $J \notin P$, we have that at least one of $\sum_{j\in J}v_j$ or $\sum_{j\in J} w_j$ will be nonzero. Thus $J$ is in the factor poset for the given collection if and only if \emph{both} $\mb{v},\mb{w}\in \h_J$. Therefore, to construct a general frame in $\R^2$ for $P$, choose two points $\mb{v},\mb{w} \in K^\perp$ such that no single forbidden hyperplane $\h_J$ for $J\notin P$ contains both $\mb{v}$ and $\mb{w}$. Every $\R^2$ frame having $P$ as its factor poset can be constructed this way.
\end{obs}

\begin{example}
Let $I=\{1,2,3\}$ and $P=\{\emptyset,\{1,2\}\}$. Then
$$2^I\setminus P =\{\{1\},\{2\},\{3\},\{1,3\},\{2,3\},\{1,2,3\}\}.$$
It is easy to check that $K =\spandex{P} =\left \{\vect{a\\a\\0}:a\in\R\right\}$ and $K^\perp = \Span\left\{\vect{0\\0\\1},\vect{1\\-1\\0}\right\}$, a plane in $\R^3$. Clearly $\dim(K)=1<3$. The forbidden hyperplanes are $\hb{1}=\hb{2} = \set{\vect{0\\0\\a}:a\in\R}$, $\hb{3} =\hb{1,2,3} = \set{\vect{a\\-a\\0}:a\in\R}$, $\hb{1,3}=\set{\vect{a\\-a\\-a}:a\in\R}$, and $\hb{2,3} = \set{\vect{a\\-a\\a}:a\in\R}$.
\begin{enumerate}
\item If we select $\mb{v} = \vect{2\\-2\\1}\in K^\perp\setminus \bigcup_{J\notin P}\h_J$ then $P$ corresponds to the frame $\{\sqrt{2}e_1,\sqrt{2}e_2,e_1\}$.

\item If we select $\mb{v}=\vect{1\\-1\\-1}$ and $\mb{w} = \vect{1\\-1\\1}$, where $\mb{v}\in \hb{1,3}$ and $\mb{w}\in \hb{2,3}$, then we get a frame $\{f_1,f_2,f_3\}\sub\R^2$ whose diagram vectors are
$$\diag{f}_1= \vect{1\\1},\diag{f}_2=\vect{-1\\-1},\diag{f}_3=\vect{-1\\1}.$$
\end{enumerate}
\end{example}

\begin{remark}
Observe that our inverse construction in principle will work in $\Real^n$, in the sense that any collection of $n(n-1)$ vectors which all lie in $K^\perp$ and do not all lie in the same forbidden hyperplane $\h_J$ will determine a collection of $k$ vectors in $\Real^{n(n-1)}$ which sum to 0 precisely on those subsets in $P$.  Further, all such subsets are constructed in this way.  However, given that the the diagram vector map from vectors to diagram vectors is not surjective for $n>2$ \cite[Remark 2.6]{tfs}, these vectors may not be diagram vectors and therefore it may be impossible to invert these to produce frame vectors in $\R^n$.  A deeper understanding of the structure of the diagram vector map may provide insights into answering this problem. In Section \ref{sect:projections} we discuss projections and provide comments as to how dilations of frames may be useful in solving a general inverse problem.
\end{remark}

For completeness we describe a method for computing inverse frames in $H_n$ for a given poset $P$. We restrict our attention to $\R^n$ as the general algorithm is similar. Given a poset $P\sub 2^I$, where $I=\{1,\ldots,k\}$, we inquire whether there exist $\fI\sub\R^n$ so that
$$\sum_{j\in J} \widetilde{f}_j = 0\text{ if and only if } J\in P.$$
In essence we are inquiring whether there exists some point in the real algebraic variety
$$\{(f_1,\ldots,f_k)\in \R^{n\times k}: \sum_{j\in J} \fdiag_j = 0 \;\forall J\in P\}$$
which does not lie in the variety
$$\bigcup_{\substack{J\notin P\\i\in\{1,\ldots,n(n-1)\}  }}\left\{(f_1,\ldots,f_k)\in\R^{n\times k}: \sum_{j\in J} \fdiag_j(i) = 0\right\}.$$
We claim such a problem can be written as a so-called \emph{nonconvex quadratically-constrained program}, i.e., a (nonconvex) system of quadratic inequalities. We summarize this in the following proposition:

\begin{proposition}
For a given poset $P\sub 2^I$, where $I= \{1,\ldots,k\}$, determining whether there exists a frame $\fI\sub\R^n$ with factor poset $P$ is equivalent to finding a solution feasible to the following system of quadratic inequalities:
$$\left\{\begin{array}{l}
\sum_{j\in J}\fdiag_j=0\;\forall J\in P\\
\sum_{i=1}^{n(n-1)} \left(r_{Ji}^++r_{Ji}^-\right)\geq1\;\forall J\notin P\\
\sum_{j\in J} \fdiag_j(i) = r_{Ji}^+-r_{Ji}^-\;\forall J\notin P,\; i =1,\ldots,n(n-1)\\
r_{Ji}^+,r_{Ji}^-\geq 0\;\forall J\notin P,\; i =1,\ldots,n(n-1)\\
r_{Ji}^+r_{Ji}^-=0\;\forall J\notin P,\; i =1,\ldots,n(n-1).
\end{array}\right.$$
\end{proposition}

\begin{proof}
It suffices to show that we can represent $\sum_{j\in J} \fdiag_j \neq0$ for $J\notin P$. This is equivalent to
$$\sum_i \left|\sum_{j\in J} \fdiag_j(i)\right| >0.$$
By multiplying all numbers by an appropriate scalar we can express this equivalently as
$$\sum_i \left|\sum_{j\in J} \fdiag_j(i)\right| \geq1.$$
Note that if $x\in\R$ then $|x|$ can be represented exactly using quadratic inequalities: if $x= r^+-r^-$ with $r^+,r^-\geq 0$ and $r^+r^-=0$ then $|x| = r^++r^-$. The desired quadratic representation given in the proposition follows.
\end{proof}
Observe that this system is prohibitively large for all but small examples. Solving nonconvex quadratic inequalities is in general a difficult problem. Such problems can be solved using a variety of branch-and-bound algorithms \cite{linderoth,sahinidis,apogee} or can be approached using heurustic methods adapted from nonlinear programming \cite{bertsekasNLP}, such as an augmented Lagrangean procedure \cite{bertsekasLag} (e.g. ADMM \cite{boyd}).

\subsubsection{Full spark frames and inverses}

A type of frame of particular interest is full spark frames, which have been studied extensively and arise in important applications such as compressive sensing \cite{fullspark}.

\begin{definition}
Let $F=\fI\sub H_n$ be a frame. We define the \emph{spark} of a frame $F$ to be the size of the smallest linearly dependent subset. We say $F$ has \emph{full spark} if $F$ has spark $n+1$, i.e., every subset of $n$ vectors of $F$ is a basis for $H_n$.
\end{definition}

As a consequence of Observation \ref{obs:rowsGiveDiagVecs} we have the following proposition.
\begin{proposition}
Suppose $P$ is the factor poset for some frame in $\R^2$. Then if $\dim(\spandex{P}) = |I|-1$, where $I$ is the index set for $P$, there exists no full spark frame in $\R^2$ with $P$ as its factor poset.
\end{proposition}

\begin{proof}
If $K=\spandex{P}$ and $\dim(\spandex{P}) = |I|-1=k-1$ then $K^\perp$ is 1-dimensional, and therefore any $\mb{v},\mb{w}\in K^\perp$ are collinear. Hence, the diagram vectors
$$\colvec{2}{v_1}{w_1}, \cdots, \colvec{2}{v_k}{w_k}$$
are collinear. Because a frame in $\R^2$ is full spark if and only if no two of its diagram vectors point along the same ray (the nonnegative span of a single vector), whenever $k\geq 3$ this means that a frame with factor poset $P$ would necessarily not have full spark.
\end{proof}

Now we give a necessary and sufficient condition for the existence of a full spark frame in $H_2\setminus\{0\}$ whose factor poset coincides with a given poset $P$.

\begin{theorem}
Let $I=\{1,2,\ldots,k\}$ be some fixed index set and $P\sub 2^I$ be a poset ordered by set inclusion and which contains no singletons. Then there exists  a full spark frame $F = \fI\sub \R^2\setminus\{0\}$ with factor poset $P$ if and only if $P$ is span-closed and $(e_i-\alpha e_j)\notin \spandex{P}$ for $i,j\in I$, $i\neq j$, and $\alpha>0$.
\end{theorem}

\begin{proof}
We begin with the forward direction. From Theorem \ref{thm:r2inverse} it follows that $P$ must be span-closed for the existence of a frame in $\R^2\setminus\{0\}$ with factor poset $P$. Suppose there exists some $i,j\in I$ with $i\neq j$ and $\alpha>0$ such that $(e_i-\alpha e_j)\in\spandex{P}$. Note that every $\mb{a}\in K^\perp$ where $K=\spandex{P}$ is orthogonal to $(e_i-\alpha e_j)$. Thus any choice of vectors $\mb{v},\mb{w}\in K^\perp$ such that no single forbidden hyperplane $\h_J$ for $J\notin P$ contains both $\mb{v}$ and $\mb{w}$ will produce diagram vectors that will satisfy $\diag{f}_i=\alpha\diag{f}_j$. This will mean the frame vectors $f_i$ and $f_j$ are collinear. Hence the frame cannot be full spark.

For the reverse direction let $K=\spandex{P}$. If $P$ is span-closed then from Theorem \ref{thm:r2inverse} there exists a frame $F=\fI\sub \R^2\setminus\{0\}$ with factor poset $P$. Suppose $(e_i-\alpha e_j)\notin K$ for every $\alpha>0$ and $i,j\in I$, $i\neq j$. Then $(e_i-\alpha e_j)^\perp\cap K^\perp\subsetneq K^\perp$. If $\mb{v}\in\R^k\cap K^\perp$ has all nonzero components then for each $i,j\in I$ there is exactly one $\alpha>0$ such that $\mb{v}\in (e_i-\alpha e_j)^\perp\cap K^\perp$. Thus for a given $\mb{v}$ with all nonzero components there can be only $\binom{k}{2}$ hyperplanes of the form $(e_i-\alpha e_j)^\perp \cap K^\perp$ containing $\mb{v}$. If we select $\mb{v}$ first, then we select $\mb{w}\in\R^k\cap K^\perp$ with all nonzero components so that $\mb{w}$ is not in any of the finite number of hyperplanes of the form $(e_i-\alpha e_j)^\perp\cap K^\perp$ containing $\mb{v}$ nor any of the forbidden hyperplanes $\h_J$ for $J\notin P$. Such a choice of $\mb{v}$ and $\mb{w}$ produces a full spark frame since the diagram vectors do not satisfy $\diag{f}_i=\alpha\diag{f}_j$ for any $i,j\in I$, $i\neq j$, and $\alpha>0$. This completes the proof.
\end{proof}

\subsection{Necessary Conditions for Factor Posets}

In this section we examine combinatorial conditions that are necessary for a poset to be a factor poset.

\subsubsection{Closure Condition}

Consider a factor poset $P$ containing the elements $\{1,2,3,4\}$ and $\{3,4,5,6\}$, and let the frame for this poset be denoted $\f$.  We know of course that the corresponding diagram vectors satisfy $$\diag{f_1} + \diag{f_2} + \diag{f_3} + \diag{f_4} = \diag{f_3} + \diag{f_4} + \diag{f_5} + \diag{f_6}$$ since both sums are equal to zero.  By subtracting $\diag{f_3} + \diag{f_4}$ from both sides of this equation, we obtain that $$\diag{f_1} + \diag{f_2} = \diag{f_5} + \diag{f_6}.$$  In this case, we say that the sets $\{1,2\}$ and $\{5,6\}$ are \emph{copies }of each other.  

\begin{definition}
For a given poset $P$ and any $A,B\in P$, then we say that $A \setminus B$ and $B \setminus A$ are \emph{copies} of one another.  
\end{definition}

\begin{remark}
For a given frame $\f$ whose factor poset is $P$, if $J$ and $K$ are copies of each other then $$\sum_{i \in J} \diag{f_i} = \sum_{i \in K} \diag{f_i}.$$  
\end{remark}

Because copies have the same sum of diagram vectors, if we remove a set from a tight frame and replace it with a copy we will not affect the tightness of that frame.  This is formalized in the following proposition.  We omit its proof.  

\begin{proposition}[Closure condition] \label{prop:closure}
For any factor poset $P$, if $J$ and $K$ are copies of each other, and there exists an element $A\in P$ such that $J \subseteq A$ and $K \cap A = \emptyset$, then $(A \setminus J) \cup K \in P$.  
\end{proposition}

\begin{example}
Let $\{1,2,3\}, \{3,4\},$ and $\{1,2,5\}$ be elements of a factor poset $P$.  Then $J = \{1,2\}$ and $K = \{4\}$ are copies, and by the closure condition if $A = \{1,2,5\}$ then $\{4,5\}$ must also be an element of $P$.  
\end{example}

It is a previously known result \cite[Prop. 3.6]{2012report} that for any factor poset $P$, and any two elements $A,B\in P$, the following are equivalent: 

\begin{enumerate}
\item $A \cup B \in P$
\item $A \cap B \in P$
\item $A \setminus B \in P$
%\item $A \triangle B \in P.$
\end{enumerate}

\begin{obs}
The equivalence of conditions $(1)$, $(2)$, and $(3)$ above follows from the closure condition.  
\end{obs}

\begin{proof}
Assume that $A \cup B \in P$.  By taking the relative complements of $A \cup B$ with $B$, we see that $A \setminus B$ and $\emptyset$ are copies.  The closure condition then says that, since $\emptyset \in P$ necessarily, $(\emptyset \setminus \emptyset) \cup (A \setminus B) = A \setminus B$ is in $P$.  Thus $(1)$ implies $(3)$.  

Assume that $A \cap B \in P$.  By taking the relative complements of $A \cap B$ with $A$, we see that $A \setminus B$ and $\emptyset$ are copies.  The closure condition then says that $(B \setminus \emptyset) \cup (A \setminus B) = A \cup B$ is in $P$.  Thus $(2)$ implies $(1)$.  

Assume that $A \setminus B \in P$.  By taking the relative complements of $A \setminus B$ with $A$, we see that $A \cap B$ and $\emptyset$ are copies.  The closure condition then says that $(\emptyset \setminus \emptyset) \cup (A \cap B) = A \cap B$ is in $P$.  Thus $(3)$ implies $(2)$.  This is sufficient to show that $(1)$, $(2)$, and $(3)$ are equivalent.  
\end{proof}

\begin{obs}
Note that for certain posets $P$ computing $\spandex{P}$ can be completed only using the empty cover. We claim that whenever $P$ satisfies the closure condition $(C)$ from Proposition \ref{prop:closure}, we have that $\spandex{P} = \Span\{[J]:J\in EC(P)\}$. This is because any poset which satisfies $(C)$ is generated by its empty cover \cite[Corr. 3.9]{2012report}. Note that for $J_1,J_2\in P$, we have $J_1\cup J_2 \in P$ iff $J_1\cap J_2\in P$. Yet, if $J_1\cap J_2\in P$, then we find that $[J_1\cup J_2] = [J_1]+[J_2]-[J_1\cap J_2]$, which proves the claim because then
$$\Span(\{[J]:J\in P\}) = \Span(\{[J]:J\in EC(P)\}).$$
\end{obs}

\subsubsection{Sign condition}

\begin{definition}
A $\emph{signing}$ of a frame $F = \fI$ is any function from $I$ to the set $\{+,-\}$ with the property that every tight subframe of $F$ contains at least one element with positive sign and at least one element with negative sign.  

A \emph{signing} of a poset $P$ with index set $I$ (by which we mean $I = \bigcup_{A\in P} A$) is any function from $I$ to the set $\{+,-\}$ with the property that every non-empty element of $P$ contains at least one index with positive sign and at least one index with negative sign.  
\end{definition}

\begin{obs}
If $F$ is a frame and $P$ its factor poset, the set of all signings for $F$ is equal to the set of all signings for $P$.  
\end{obs}

Given a frame $F = \f\sub\R^n$, consider the set of diagram vectors $\diagf \sub \R^{n(n-1)}$.  Any codimension 1 hyperplane $\Gamma$ in $\R^{n(n-1)}$ divides the space into two halfspaces $\Gamma^+$ and $\Gamma^{-}$. These are determined (up to an overall sign) by any nonzero $\tau\in \Gamma^\perp$. For any such $\tau$ we let
$$\Gamma^+ = \{x\in \R^{n(n-1)}:\langle x,\tau\rangle \geq0\}\text{\quad and \quad} \Gamma^-=\{x\in \R^{n(n-1)}:\langle x,\tau\rangle \leq0 \}.$$
Equivalently, given a vector $\tau\in\R^{n(n-1)}$ we can divide $\R^{n(n-1)}$ into two regions by considering the sign of the inner product of a given vector $x\in \R^n$ with $\tau$.  If none of the diagram vectors for $F$ lie on $\Gamma$ (equivalently none are orthogonal to $\tau$), then this is a signing of $F$.  

\begin{proposition} \label{linesgetsignings}
For a given frame $F = \fI$, and a vector $\tau\in\R^{n(n-1)}$ which is not orthogonal to any diagram vector of $F$, the function $i \mapsto \operatorname{sign}(\inpro{\diag{f}_i, \tau})$ is a signing of $F$.  
\end{proposition}

\begin{proof}
Suppose that we have a tight subframe $\fJ$ of $F$ and a vector $\tau$ which is not orthogonal to any diagram vector for $F$. Therefore, $\sum_{i \in J} 
\diag{f}_i = 0$ since this subframe is tight.  But the inner product of this sum with $\tau$ is $\sum_{i \in J} \inpro{\diag{f}_i,\tau}$.  We know by assumption that none of these summands is zero, and the sum is non-empty, so if $\fJ$ does not contain at least one element of positive sign and at least one element of negative sign, then this is a set of real numbers of all the same sign which add up to $\inpro{0,\tau} = 0$, a contradiction.  

\end{proof}

We now present a number of immediate corollaries of Proposition \ref{linesgetsignings}. Throughout we assume the poset $P$ satisfies $P\sub 2^I$.

\begin{corollary}[Sign Condition]
If a poset $P$ does not have any signings, then it is not a factor poset for any frame (with all non-zero vectors).  
\end{corollary}

\begin{corollary} \label{cor:parasign}
Let $P$ be a poset and $i,j$ two of its indices.  If $\sigma(i) = \sigma(j)$ for every signing $\sigma$ of $P$ then every frame $\fI$ with $P$ as its factor poset will have $f_i$ and $f_j$ collinear.  
\end{corollary}

\begin{proof}
If there exists a frame with $f_i$ and $f_j$ not collinear, then the diagram vectors $\diag{f_i}$ and $\diag{f_j}$ will not be parallel.  Thus there exists some hyperplane with $\diag{f_i}$ and $\diag{f_j}$ on opposite sides and which does not contain any of the diagram vectors.  This produces a signing $\sigma$ with $\sigma(i) \neq \sigma(j)$.  The result follows by contraposition.  
\end{proof}

\begin{corollary} \label{cor:antiparasign}
Let $P$ be a poset and $i,j$ two of its indices.  If $\sigma(i) \neq \sigma(j)$ for every signing $\sigma$ of $P$ then every frame $\fI$ with $P$ as its factor poset will have $\diag{f_i}$ and $\diag{f_j}$ antiparallel. 
\end{corollary}

Here by antiparallel we mean the following: $\diag{f}_i$ and $\diag{f}_j$ are antiparallel iff there exists some $\alpha<0$ so that $\diag{f}_i = \alpha \diag{f}_j$. We now proceed with the proof.

\begin{proof}
First observe that $\diag{f}_i$ and $\diag{f}_i$ cannot be parallel, where by parallel we mean that one is a positive scalar multiple of the other. If this were the case then certainly there is a signing with $\sigma(i)=\sigma(j)$. We now proceed with the proof of the claim. Assume there exists a frame $\fI$ with $\diag{f_i}$ and $\diag{f_j}$ not collinear. Consider the set
$$A = \{x\in\R^{n(n-1)}:\langle x,\diag{f}_i\rangle >0\}\cap \{x\in\R^{n(n-1)}:\langle x,\diag{f}_j\rangle >0\},$$
a full-dimensional unbounded (open) polyhedron in $\R^{n(n-1)}$ (see \cite{ziegler}). Because $A$ is full-dimensional, $A\setminus\bigcup_{i\in I}\diag{f}_i^\perp$ is non-empty, and hence any $\tau$ in this set which is not orthogonal to any diagram vectors gives a signing $\sigma$ with $\sigma(i) = \sigma(j)$.  The result follows by contraposition.   
\end{proof}

\begin{corollary} \label{collinearsign}
Let $P$ be a poset, and let $F$ be a frame with $P$ as its factor poset.  If there exists a unique signing for $P$ (up to an overall change of sign), then all the diagram vectors for $F$ must be collinear.  Equivalently, all of the vectors in $F$ lie along two orthogonal lines.  
\end{corollary}

\begin{proof}
All of the diagram vectors for elements of $F$ can be divided into two groups $A$ and $B$, so that $\diag{f_i}$ is in $A$ if $\sigma(i) = +$ and in $B$ otherwise.  By Corollaries \ref{cor:parasign} and \ref{cor:antiparasign}, we know that all of the diagram vectors in $A$ are parallel to each other, all those in $B$ are parallel to each other, and all diagram vectors in $A$ are antiparallel to those diagram vectors in $B$.  Taken together, this means that every diagram vector for $F$ is collinear with every other diagram vector for $F$.  
\end{proof}

\begin{corollary} \label{cor:signspark}
Let $P$ be a poset.  If there exist two indices $i,j$ in the index set $I$ for $P$ such that every signing $\sigma$ of $P$ has $\sigma(i) = \sigma(j)$, then any frame $F$ with $P$ as its factor poset has spark at most $2$.  
\end{corollary}

\begin{proof}
By Corollary \ref{cor:parasign}, the diagram vectors for $f_i$ and $f_j$ are parallel.  This can only happen if $f_i$ and $f_j$ are collinear, so the spark of $F$ is less than or equal to 2.  
\end{proof}

\subsection{Projections of factor posets}\label{sect:projections}

Here we address how factor poset structure is preserved under (orthogonal) projections. The necessity of span-closure for the inverse problem (Theorem \ref{thm:inverse}) along with its sufficiency for $\R^2$ gives the following result:

\begin{proposition}\label{prop:nto2Reduct}
Given a frame $F=\fI\sub H_n$ there exists a frame $G=\gI\sub \R^2$ such that $F$ and $G$ have the same factor poset.
\end{proposition}

Therefore, it is natural to inquire whether given a factor poset for a frame $F$ in $H_n$, for all $\ell\in\N$ with $2\leq \ell\leq n$ does there exist a frame $G\sub H_\ell$ with the same factor poset as $F$? Here we answer in the affirmative by way of studying the structure of non-tight frames under projections. The properties of projections of frames have been studied in several different contexts \cite{project,memoir}. The projection of any tight frame is tight, and if $F$ is a frame with frame bounds $0<A\leq B<\infty$, then any projection $\pi F$ is a frame with lower frame bound at least $A$ and upper frame bound at most $B$ (the bounds $A$ and $B$ are valid frame bounds, although they may not be optimal after projection). For example, projecting the (non-tight) frame $\{e_1,e_2,3e_3\}\sub \R^3$ onto $e_3^\perp$ gives a tight frame for $\R^2$, whereas projecting onto $e_2^\perp$ does not give a tight frame. 

Consider the problem of projecting a frame in $H_n$ to a frame in $H_{n-1}$ with the same factor poset structure. Any tight subframe $\fJ\sub\fI$ will certainly remain tight. To project via $\pi$ to a frame $\pi F\sub H_{n-1}$ with the same factor poset we must also have that $\pfJ$ is not tight (in $H_{n-1}$) whenever $\fJ$ is not tight (in $H_n$). Therefore, we must inquire when a non-tight frame is projected to a non-tight frame. For a frame $F=\fI\sub H_n$ with frame operator $S$, a natural function to study is
$$\Lambda_F(x) :=\sum_{i\in I} |\langle x,f_i\rangle|^2=\langle Sx,x\rangle\quad\text{ for } x\in \Us^{n-1}=\{x\in H_n:\|x\|_2=1\},$$
since $\min_{x\in \Us^{n-1}}\Lambda_F(x)$ and $\max_{x\in \US^{n-1}} \Lambda_F(x)$ are the optimal lower and upper frame bounds, respectively, for a frame $F$. In the following proposition we use the structure of this function to answer our desired question on factor-poset-preserving projections. The results are slightly different whether $H_n$ is $\R^n$ or $\C^n$. Our proof uses an interlacing inequality on eigenvalues of projections.

\begin{lemma}[\cite{project}]\label{lemma:projectionInterlacing}
Let $\fI\sub H_n$ be a frame with frame operator $S$ and $\pi$ be a rank $\ell\leq n$ orthogonal projection on $H_n$. Let $\lambda_1\geq \lambda_2\geq\cdots\geq \lambda_n$ be the eigenvalues for $S$ and let $\mu_1\geq \mu_2\geq\cdots\geq \mu_\ell$ be the eigenvalues for the frame operator $\pi S\pi$ of $\{\pi f_i\}_{i\in I}$. Then for all $1\leq i \leq \ell$ we have
$$\lambda_i \geq \mu_i \geq\lambda_{n-\ell+i}.$$
In particular, when $\pi$ is a rank $n-1$ projection, we have
$$\lambda_1\geq \mu_1 \geq \lambda_2 \geq\mu_2\geq \cdots \geq \lambda_{n-1}\geq \mu_{n-1}\geq\lambda_n.$$
\end{lemma}

\begin{lemma}\label{lemma:projectSameFrameBound}
Suppose $F=\fI\sub H_n$ is a frame, where $n\geq 3$, and let $\pi,\rho:H_n\to H_n$ be two rank $n-1$ projections. If $\pi F$ and $\rho F$ are both tight frames for $\pi(H_n)$ and $\rho(H_n)$, respectively,  then they have the same frame bound.
\end{lemma}
\begin{proof}
Note that for a  projection $\pi$ onto a $(n-1)$-dimensional subspace $K\sub H_n$ we have
$$\Lambda_{\pi F}(x) = \langle \pi S \pi x,x\rangle = \langle S\pi x,\pi x\rangle = \langle Sx,x\rangle = \Lambda_{F}(x)\text{\quad for }x\in \US^{n-1}\cap K.$$
From this observation we note that $\pi F$ is a tight frame for $K$ if and only if $\Lambda_{\pi F}$ is constant on $\US^{n-1}\cap K$ if and only if $\Lambda_F$ is constant on $\US^{n-1}\cap K$.

Now let $K_\pi = \pi(H_n)$ and $K_\rho = \rho(H_n)$, both of which are $n-1$ dimensional subspaces of $H_n$. Because $K_\pi$ and $K_\rho$ intersect nontrivially, there is a subspace $M=K_\pi\cap K_\rho\sub H_n$ of positive dimension for which $\Lambda_F$ is constant on $M\cap \US^{n-1}$. Because $\Lambda_F$ is constant on $M$, $K_\pi$, and $K_\rho$ (restricted to the unit sphere), we must have that $\Lambda_F$ is the same constant on $K_\pi$ and $K_\rho$. Hence, $\pi F$ and $\rho F$ have the same frame bound.
\end{proof}

We now prove our desired result on projections of non-tight frames.

\begin{proposition}\label{prop:nontightProjectReal}
Let $\fI\sub \R^n$ be a frame which is not tight, where $n\geq 3$. Then for all but at most two projections $\pi $ onto $(n-1)$-dimensional subspaces of $\R^n$, $\{\pi f_i\}_{i\in I}$ is not a tight frame for $\pi(\R^n)$. Equivalently, there are at most two rank $n-1$ projections for which $\{\pi f_i\}_{i\in I}$ is a tight frame for $\pi(\R^n)$.
\end{proposition}

\begin{proof}
Let $S$ be the frame operator for $F=\fI\sub \R^n$. Since $F$ is not tight, $\Lambda_F$ is not constant on $\US^{n-1}$. By the observation in the proof of Lemma \ref{lemma:projectSameFrameBound}, to prove the desired claim it suffices to show there are only finitely many $(n-1)$-dimensional subspaces $K\sub \R^n$ so that $\Lambda_F$ is constant on $K\cap \US^{n-1}$. 

To proceed, let $\lambda_1\geq \lambda_2\geq \cdots\geq \lambda_n>0$ be the eigenvalues for $S$ with corresponding orthonormal eigenvectors $\{\eta_1,\ldots,\eta_n\}$. For $x\in \US^{n-1}$ we may write $x = \sum_{i=1}^n \langle x,\eta_i\rangle\eta_i$ with $1 = \sum_{i=1}^n |\langle x,\eta_i\rangle|^2$ by Parseval's identity. Therefore,
$$\Lambda_F(x) = \langle Sx,x\rangle = \sum_{i=1}^n \lambda_i |\langle x,\eta_i\rangle|^2.$$
We first consider the case where $n\geq 4$. Observe that if $\lambda_2 > \lambda_{n-1}$ then no rank $n-1$ projection of $F$ can result in a tight frame for $\pi(\R^n)$ by the interlacing inequality from Lemma \ref{lemma:projectionInterlacing} (note here why we restrict our attention to $n\geq 4$). Therefore for what remains we shall assume that $\lambda:=\lambda_2=\lambda_3=\cdots=\lambda_{n-1}$. If we require that $\Lambda_F$ is constant on $K\cap \US^{n-1}$ then by the interlacing lemma this constant must be equal to $\lambda$. Hence we consider the deviation of $\Lambda_F$ from $\lambda$:
\begin{align*}
\Lambda_F(x)-\lambda &= \sum_{i=1}^n \lambda_i |\langle x, \eta_i\rangle |^2 - \lambda\\
&=\sum_{i=1}^n \lambda_i |\langle x, \eta_i\rangle |^2 - \lambda\sum_{i=1}^n|\langle x,\eta_i\rangle|^2\\
&= (\lambda_1-\lambda)|\langle x,\eta_1\rangle|^2 - (\lambda-\lambda_n)|\langle x,\eta_n\rangle|^2.
\end{align*}
Let $a = \lambda_1-\lambda\geq0$ and $b = \lambda-\lambda_n\geq 0$. Note that since $F$ is not tight, at least one of $a$ and $b$ is nonzero. We now consider two cases:
\begin{enumerate}
\item Suppose $\lambda_1 > \lambda>\lambda_n$, i.e., $a>0$ and $b>0$. Then, if for $x\in \US^{n-1}$ we have $\Lambda_F(x) -\lambda  = 0$ this implies
$$|\langle x,\eta_1\rangle|^2 = \frac{b}{a}|\langle x,\eta_n\rangle|^2.$$
The set of $x\in \R^n$ satisfying this equation is the union of two codimension $1$ hyperplanes, namely $\left(\eta_1\pm \frac{a}{b}\eta_n\right)^\perp$, and hence there are two projections which result in tight frames. 

\item Suppose $\lambda_1 > \lambda = \lambda_n$, so $a>0=b$. Here
$$\Lambda_F-\lambda = a|\langle x,\eta_1\rangle|^2,$$
which can only equal zero for all $x\in \US^{n-1}\cap K$ where $\dim(K)=n-1$ when $K= \eta_1^\perp$. Therefore there is one and only one projection which results in a tight frame. The case for $a=0<b$ is essentially identical.
\end{enumerate}

In the case where $n=3$, we know by \cite[Prop. 3.5]{project} that there is a rank $2$ projection $\pi$ which makes $\pi F$ tight in $\pi(H_3)$ with frame bound $\lambda_2$. By the preceding lemma we know that any projection which is tight must therefore have frame bound $\lambda$. Therefore, as before we consider
$$\Lambda_F(x)-\lambda_2 =(\lambda_1-\lambda_2)|\langle x,\eta_1\rangle|^2 - (\lambda_2-\lambda_3)|\langle x,\eta_3\rangle|^2.$$
As before we have that the solutions either lie in two $2$-dimensional subspaces (when $\lambda_1>\lambda_2>\lambda_3$) or lie in one $2$-dimensional subspace (when $\lambda_1=\lambda_2$ or $\lambda_2=\lambda_3$).

The desired result follows, as there are 0, 1, or 2 rank $n-1$ projections of $\R^n$ for which $\pi F$ is tight in $\pi(\R^n)$.
\end{proof}

\begin{proposition}\label{prop:nontightProjectComplex}
Let $\fI\sub \C^n$ be a frame which is not tight, where $n\geq 3$. Let $V\sub \US^{n-1}\sub \C^n$ denote the set of unit vectors $f\in\C^n$ so that the rank $n-1$ projection $\pi_f$ onto $f^\perp$ results in a tight frame $\pi F$ for $\pi(\C^n)$. Then $V$ is entirely contained in a subspace of $\C^n$ of dimension $2$.
\end{proposition}

\begin{proof}
The proof carries through almost entirely as before, except Case 1 (where $\lambda_1>\lambda_2=\cdots=\lambda_{n-1}>\lambda_n$) deviates from before. Consider again the equation of the form
$$|\langle x,\eta_1\rangle|^2 = c |\langle x,\eta_n\rangle |^2,$$
where $c>0$ is a real number. Without loss of generality we take $c=1$. The set of solutions $x\in\C^n$ to this equation is precisely
$$\bigcup_{\delta\in[0,2\pi)} \left(\eta_1+e^{i\delta}\eta_n\right)^\perp,$$
a union of hyperplanes. The collection $V$ in this case (where $\lambda_1>\lambda > \lambda_n$) is precisely
$$V = \left\{\frac{e^{i\delta_1}}{\sqrt{2}}\left(\eta_1+e^{i\delta_2}\eta_n\right):\delta_1,\delta_2\in[0,2\pi)\right\}.$$
Note the necessity of $e^{i\delta_1}$ since a rank $n-1$ projection determines a vector orthonormal to its image uniquely up to a scalar of modulus $1$. We have that $V$ is contained in $\Span\{\eta_1,\eta_2\}$, a two-dimensional subspace in $\C^n$. In Case 2 (where either $\lambda_1=\lambda_2$ or $\lambda_n = \lambda_2$), the same conclusion as before holds.
\end{proof}

\begin{remark}
Observe that since the set $V\cap \US^{n-1}$ lies in a subspace of dimension $2<n$, its Lebesgue measure on the sphere $\US^{n-1}$ is zero. Hence, with probability one, a rank $n-1$ projection chosen uniformly at random will project a non-tight frame to a non-tight frame. Therefore, given a frame $F\sub H_n$ we have that for almost all rank $n-1$ projections $\pi$ on $H_n$ that $\pi F$ has the same factor poset as $F$.
\end{remark}

An immediate consequence of these propositions is the main theorem for this subsection. This can be seen as a strengthening of Proposition \ref{prop:nto2Reduct}.

\begin{theorem}
Suppose $F=\fI\sub H_n$ is a frame. Then for any $2\leq\ell\leq n$ there exists some rank $\ell$ projection $\pi:H_n\to H_n$ such that $\pi F\sub \pi (H_n)$ has the same factor poset as $F$.
\end{theorem}

\begin{corollary}
Let $P\sub 2^I$ be a span-closed poset which contains no singletons and suppose $\{\emptyset\}\subsetneq P$. Then there exists some $m\in \N$, $m\geq 2$, so that an $H_\ell$ inverse frame exists for all $\ell \in\N$ with $2\leq \ell\leq m$ and no $H_\ell$ inverse frame exists for any $\ell > m$.
\end{corollary}

Hence, the largest possible dimension for the existence of an inverse frame is an intrinsic property of the factor poset. We believe that it may be possible to solve the general inverse problem by inquiring when it is possible to lift a frame in $\R^n$ to a frame in $\R^{n+1}$ with the same factor poset. This can be seen as restricted version of a dilation problem (see \cite{memoir} for commentary on dilations of frames). This problem remains open for future work. We conclude this section by considering a bound on the largest possible such dimension. For a poset $P\sub 2^I$ which is span-closed, let $\calD(P)$ denote the largest dimension $d$ for which a frame in $\R^d$ exists with factor poset $P$. In the following proposition we give a trivial bound on $\calD(P)$.

\begin{proposition}\label{prop:badBoundLift}
For a span-closed poset $P\sub 2^I$ we have that
$$\calD(P)\leq \min\{ |A|:A\in P, A\neq \emptyset\}.$$
\end{proposition}
This bound is true since every frame in $H_n$ contains at least $n$ vectors. However, this bound can be arbitrarily bad in the following sense.

\begin{example}
Here we show that for any $m\in\N$ there exists a frame in $\R^2$ with a factor poset $P$ so that (1) $\min\{ |A|:A\in P, A\neq \emptyset\}\geq m$ and (2) there exists no $\R^3$ frame with $P$ as its factor poset. In this sense, the bound in Proposition \ref{prop:badBoundLift} can be arbitrarily bad. Let $e_1,e_2\in\R^2$ be the standard ONB for $\R^2$. Fix $m\in\N$ and let $k\geq 2m+2$ be any even integer. Define the frame of $k$ vectors in $\R^2$ to be
$$F = \{f_i\}_{i=1}^k=\{\underbrace{e_1,\ldots,e_1}_{k-2 \text{ times}},-\sqrt{k/2-1}e_2,-\sqrt{k/2-1}e_2\}.$$
It is straightforward to verify that the factor poset $P$ for $F$ satisfies
$$\min \{|A|:A\in P,A\neq \emptyset\} = k/2-1\geq m.$$
Further, no $\R^3$ frame can have $P$ as its poset. Suppose $G=\{g_i\}_{i=1}^k\sub H_n$ is a frame with $P$ as its factor poset. Then using the poset structure alone we see that $\diag{g}_1,\ldots,\diag{g}_{k-2}$ must all be copies of one another, and hence each element in the empty cover $EC(P)$ corresponds to a frame with $k/2-1$ copies of one vector and one other vector. Hence, these correspond to frames with $2$-dimensional span, and therefore no $\R^3$ inverse can exist. Thus, $\calD(P)=2$.
\end{example}

One interesting question is to study how the function $\calD$ behaves under the intersection of posets. By earlier remarks on span-closure we know that if $P_1$ and $P_2$ are factor posets for frames in $H_n$ and $H_m$, respectively, then $P_1\cap P_2$ is the factor poset for a frame in $\R^2$, and hence $\calD(P_1\cap P_2)\geq 2$. We now present an example which highlights some of the behavior of $\calD$.

\begin{example}%%% Intersection of two gives prime.
Fix $I=\{1,\ldots,k\}$ with $k\geq 5$ and let $F=\fI,G=\gI\sub\R^2$ be two frames with respective factor posets
$$\fpF = \{\emptyset,\{1,2\},I\setminus \{1,2\},I\}\text{ and }\fpG = \{\emptyset,\{k-1,k\},I\setminus\{k-1,k\},I\}.$$
Such frames exist by \cite[Prop. 2.8]{prime} or by an elementary argument using diagram vectors. Note that $\calD(\fpF)=\calD(\fpG)=2$. However,
$$\calD(\fpF\cap \fpG) = \calD(\{\emptyset,I\}) =|I| = k.$$
Here we have implicitly used the existence of prime tight frames with $k$ vectors in $H_n$ for any $k\geq n$ \cite[Prop. 2.8]{prime}.
\end{example}

We believe that that while $\calD(P_1\cap P_2)$ cannot be bounded above easily, as hinted at in the previous example, the following lower bound holds.

\begin{conjecture}
Let $P_1$ and $P_2$ be two factor posets. Then
$$\calD(P_1\cap P_2)\geq \min\{\calD(P_1),\calD(P_2)\}.$$
\end{conjecture}

Based on the results presented above, to prove the conjecture it suffices to show that for two frames $\fI,\gI\sub\R^n$ with factor posets $P_1$ and $P_2$, respectively, there exists some $\{h_i\}_{i\in I}\sub \R^n$ with factor poset $P_1\cap P_2$.

\subsection{Factor poset enumeration and size bounds}

In this subsection we touch on two distinct problems related to factor poset enumeration for a fixed dimension $n$ and number of vectors $k$:
\begin{enumerate}
\item how many non-isomorphic (or non-strongly-isomorphic) factor posets are there for frames $\fI\sub H_n$ with $|I|=k$, and
\item how large can these factor posets be? 
\end{enumerate}
By \emph{isomorphic} we mean the usual definition of isomorphism on posets $(P_1,\leq_1)$ and $(P_2,\leq_2)$, namely, an invertible mapping $\phi:P_1\to P_2$ so that for all $a,b\in P_1$ we have
$$\phi(a)\leq_2\phi(b)\text{ if and only if }a\leq_1 b.$$
For two subsets $P_1,P_2\sub 2^I$ for an index set $I$, both partially-ordered by set inclusion, we say $P_1$ and $P_2$ are \emph{strongly isomorphic} if they are isomorphic and there exists some permutation $\sigma$ of the indices of $I$ so that $\sigma(P_1)=P_2$. We see that strong isomorphism classes are the collection of factor posets modulo the action of the symmetric group on the indices in $I$.

Here we begin by considering the enumerative question for frames in $\R^2$. Note that factor posets are uniquely determined by their index span (because of span-closure), hence enumerating factor posets for $\R^2$ is equivalent to enumerating the number of different subsets of $\{0,1\}^k\sub\R^k$ which are realized as the intersection of a subspace with $\{0,1\}^k$, the vertices of the unit hypercube (also known as the $0/1$-polytope). Technically speaking, we should restrict our attention to subsets which contain no vector $e_1,\ldots,e_k$ and we should only count up to permutation of the indices (given we are counting strong isomorphism). While it appears that an application of Burnside's lemma may be appropriate here, for the situation we are considering it is difficult to apply.

In general this appears to be a difficult problem \cite{aicholzer}. Here we only discuss the computability of this number. A primitive result here is as follows:

\begin{proposition}
For any $k\in \N$ there exists a brute-force enumeration scheme to determine the number of factor posts for frames with $k$ vectors in $\R^2$ which is guaranteed to terminate in time $O\left(2^{5k/2-k^2}k^{k^2/2}\right)$.
\end{proposition}

\begin{proof}
Let $P\sub 2^I$ be a factor poset for some $\R^2$ frame. Hence, $P$ satisfies the conditions of Theorem \ref{thm:r2inverse}. Let $K=\spandex{P}$. Define
$$A = K^\perp\setminus \left(\bigcup_{J\in 2^I\setminus P} \{\aaa\in K^\perp:\langle \aaa,[J]\rangle=0\}\right)\neq\emptyset,$$
the object of central interest in Theorem \ref{thm:r2inverse}. We claim that $A\cap \Z^k\neq \emptyset$. It suffices to show that $A\cap \Q^k\neq \emptyset$, which follows from a standard Cramer's rule argument; hence, we have the desired $A\cap\Z^k\neq\emptyset$. Consequently, there is some subset of $\Z\setminus\{0\}$ whose subset sum structure corresponds to the poset $P$. Now let $\mb{a}\in A\cap\Z^k$. We may assume that $\gcd(\mb{a})=\gcd(a_1,\ldots,a_k)=1$. In \cite{aicholzer} it is shown that we must have that $\|\mb{a}\|_\infty\leq 2^{-(k-1)}k^{k/2}$. Hence, we may enumerate all factor posets by computing the subset-sum structure for the $\left(2\cdot2^{-(k-1)}k^{k/2} \right)^k$ integer $k$-tuples with no zero components and $\ell_\infty$ norm at most $2^{-(k-1)}k^{k/2}$. Subset-sum structure can be computed in time $O(2^{k/2})$ \cite{horowitzsahni}, hence this brute force algorithm is
$$O\left(2^{k/2}\left(2\cdot2^{-(k-1)}k^{k/2} \right)^k\right)=O\left(2^{5k/2-k^2}k^{k^2/2}\right).$$
\end{proof}

\begin{remark}
Let us remark that this is an improvement over the ``obvious'' brute-force enumeration scheme. Namely, calculate the powerset of $\{0,1\}^k$ (which has $2^{2^k}$ elements). For each of these sets $A\sub \{0,1\}^k$ we must determine if it is span-closed. Define the matrix $M_A$ to be the $k\times 2^k$ matrix whose first $|A|$ columns are the vectors in $A$ and whose remaining columns are the vectors $\{0,1\}^k\setminus A$. Note that $M_A$ can be row-reduced to $\rowred(M_A)$ in time $O(k^3)$ (although storage is exponential in $k$). If $M_A$ is span-closed and $A$ contains no vector $e_1,\ldots,e_k$ (corresponding to a poset not containing singletons) then we can express $\rowred(M_A)$ in the block form
$$\rowred(M_A) = \begin{bmatrix}
I & M_1 \\
\mathcal{O}&M_2
\end{bmatrix},$$
where $\mathcal{O}$ is a matrix with all entries equal to zero. Then $M_A$ is span-closed if and only if $M_2$ contains no columns which are entirely zero. Hence, this brute forces scheme has time complexity $O\left(k^3 2^{2^k}\right)$, which is significantly worse than that given in the previous proposition.
\end{remark}

We now shift our attention to bounding the possible size of factor posets. Here there are two distinct questions of interest:
\begin{enumerate}
\item how many possible subsets of a given tight frame can be tight, and
\item how many possible subsets of a given tight frame can be prime tight frames?
\end{enumerate}

The first question is answered using the existing literature, where the size of subset-sum structures has been studied using techniques from complex analysis \cite{nabeya} and Sperner-like theorems in poset theory \cite{furedi}. We begin with the theorem statement.

\begin{theorem}[\cite{furedi}]
Let $A\sub\R$ be a set of $k\geq 2$ nonzero numbers which sum to zero and let $P$ be its subset-sum poset. Then
$$|P|\leq \left\{\begin{array}{cl}
\binom{k}{k/2},& k\text{ is even,}\\
2\binom{k-1}{\floor{k/2}-1},&k\text{ is odd.}
\end{array}\right.$$
Moreover, this bound is tight. Up to permutation of indices and a scaling of all numbers, these bounds are achieved uniquely by $a_1 = \cdots = a_{k/2}=1,a_{k/2+1}=\cdots=a_k = -1$ for $k$ even, and $a_1=2,a_2=\cdots=a_{\floor{k/2}-1}=1,a_{\floor{k/2}}=\cdots=a_{k}=-1$ for $k$ odd.
\end{theorem}

Using the connection of $\R^2$ factor posets to the subset-sum problem there is an immediate corollary for tight frames.

\begin{corollary}\label{corollary:fpSizeBound}
Let $F=\f\sub\Hnwozero$ be a tight frame with factor poset $\fpF$. Then
$$|\fpF|\leq  \left\{\begin{array}{cl}
\binom{k}{k/2},& k\text{ is even,}\\
2\binom{k-1}{\floor{k/2}-1},&k\text{ is odd.}
\end{array}\right.$$
Moreover, this bound is tight for $n=2$ and not tight for $n\geq3$.
\end{corollary}

This bound is not tight for $n\geq 3$ because those tight frames which realize the bound have no inverse in $H_3$ (which can be argued via elementary means). We have not been able to prove the following conjecture:

\begin{conjecture}\label{conj:fpSize}
Let $F=\f\sub\Hnwozero$ be a tight frame with factor poset $\fpF$ and suppose $|I|=k=mn$, where $m\in\N$. Then
$$|\fpF|\leq \sum_{\ell=0}^m \binom{m}{i}^n.$$
Moreover, this bound is tight taking
$$G = \{\underbrace{e_1,\ldots,e_1}_{m\text{ times}},\ldots,\underbrace{e_n,\ldots,e_n}_{m\text{ times}}\}.$$
\end{conjecture}

Observe that this conjectured bound is precisely that which is obtained for $n=2$. This is true because
$$\sum_{i=0}^m \binom{m}{i}^2 = \binom{2m}{m}.$$
Note that this proposed bound is a hypergeometric function ${}_nF_{n-1}$ for which there is no closed form in terms of binomial coefficients when $3\leq m \leq 9$ (see \cite{calkin}). Despite the multitude of literature on the structure of antichains in posets \cite{bollobas,ahlswedezhang,sperner}, we have not yet been able to apply these results for the form of constrained subset-sum problem arising in $H_n$ for $n>2$.

We now turn our attention to bounding the number of prime tight subframes for a given tight frame in $\R^2$. In other words, given a frame $F=\fI\sub\R^2$ with factor poset $\fpF$, what is the largest possible size of $|EC(\fpF)|$? We conjecture that it must be that
$$|EC(\fpF)|\leq 2 \binom{k-2}{\left\lfloor k/2-1\right\rfloor}.$$
While we have no proof that this bound is true, it is trivial to show it is asymptotically tight in $k$. Consider the construction given in the following example.

\begin{example}
Fix $k\geq 4$ and $n=2$. Define the frame of $k$ vectors in $\R^2$ to be
$$F_k = \{\underbrace{e_1,\ldots,e_1}_{k-2 \text{ times}},-\sqrt{\lfloor k/2-1\rfloor}e_2,-\sqrt{\lceil k/2-1\rceil}e_2\}.$$
It is straightforward to verify that $|EC(\mathbb{F}_{F_k})| = 2 \binom{k-2}{\left\lfloor \frac{k-2}{2}\right\rfloor}$ as the prime tight subframes consist of $\floor{k/2-1} $ choices of any of the first $k-2$ vectors paired with the penultimate vector, or any $\ceil{k/2-1}$ choices of the first $k-2$ vectors paired with the final vector. This gives
$$\binom{k-2}{\floor{k/2-1}} + \binom{k-2}{\ceil{k/2-1}} = 2 \binom{k-2}{\floor{k/2-1}}.$$
\end{example}

Let $f(k)=\Theta(g(k))$ denote that $\lim_{k\to\infty} f(k)/g(k)$ exists and is contained in $(0,\infty)$. We now show that the desired bound is asymptotically tight.
\begin{proposition}
Let $E_k$ denote the size of the largest possible empty cover for frames $F=\fI\sub\R^2\setminus\{0\}$ with $|I|=k$. Then
$$E_k = \Theta\left(\binom{k-2}{\floor{k/2-1}}\right).$$
\end{proposition}
\begin{proof}
We know by the constructive example above that $E_k \geq 2\binom{k-2}{\floor{k/2-1}}$. Further, by Corollary \ref{corollary:fpSizeBound} we have that
$E_k \leq \binom{k}{\floor{k/2}}$. One can readily verify that 
$$\lim_{k\to\infty}\left. \binom{k}{\floor{k/2}}\right/\left(2\binom{k-2}{\floor{k/2-1}}\right)=2,$$
hence $E_k = \Theta\left(\binom{k}{\floor{k/2}}\right) = \Theta\left(\binom{k-2}{\floor{k/2-1}}\right)$.
\end{proof}

The conjectured bound is therefore tight (up to the multiplicative constant $2$) in the limit as $k\to\infty$. Let us note that it is easy to show that the desired (finite $k$) bound does hold true for certain factor posets, such as where the empty cover can be partitioned into $A\sqcup B=EC(\fpF)$ so that $\left|\bigcup_{a\in A} a \right|\leq k-2$ and $\left|\bigcup_{b\in B} b \right|\leq k-2$. However, this is not true for all empty covers of $\R^2$ frames. We suspect that a modification of the proof argument in \cite{furedi} on the size of the entire factor poset may be useful in proving our desired result. Further, it may be possible to develop an Ahlswede-Zhang (AZ) type identity \cite{ahlswedezhang} for factor posets expressing the connection between the size of the factor poset and the size of its empty cover. Clearly there is a trade off between the two, and therefore an AZ-type identity may also further illuminate the underlying combinatorial structure of factor posets.

%%%%%%%%% Scalability posets

\section{Scalable Frames}
\subsection{Scalability Posets}

In this section, we discuss scalable frames and characterize when scalings with certain properties exist. Throughout our discussion of scalability we assume that $F=\{f_i\}_{i\in I}\subseteq H_n$ is a unit-norm frame, i.e., $F$ is a frame and $\|f_i\|=1\;\forall i$.

\begin{definition}[\cite{noteOnScalableFrames}]
For a frame $\F =\{f_i\}_{i\in I}\sub H_n$, a \emph{scaling} is a vector $w=(w(1),\ldots,w(k))\in \mathbb{R}_{\geq 0}^k$ such that $\{\sqrt{w(i)}f_i\}_{i\in I}$ is a Parseval frame for $H_n$. If a scaling exists, $\F$ is said to be \emph{scalable}. 
\end{definition}

\begin{definition}[\cite{noteOnScalableFrames}]
A scaling $w$ is \emph{minimal} if $\{f_i \, : \, w(i) >0\}$  has no proper scalable subset. 
\end{definition}

The set of all scalings for a frame $\F$ can be described in terms of its minimal scalings. Before proceeding, we require a few additional definitions.

\begin{definition}
Let $\{x_i\}_{i=1}^m$ be a set of points in $H_n$. A point $x$ is a \emph{convex combination} of points from $\{x_i\}_{i=1}^m$ if
$x=\sum_{i=1}^m \alpha_ix_i,$ where $\alpha_i\geq 0$ and $\sum_{i=1}^m\alpha_i=1$. The set of all convex combinations of points in $\{x_i\}_{i=1}^m$ is called the \emph{convex hull} of $\{x_i\}_{i=1}^m$, and is defined as
	\begin{equation*}
	\conv\{x_i\}_{i=1}^m:=\left\{\sum_{i=1}^m\alpha_ix_i\, :\,\alpha_i\geq 0,\,\sum_{i=1}^m\alpha_i=1\right\}.
	\end{equation*}
\end{definition}

Let $\F=\{f_i\}_{i\in I}$ be a frame and let $\mathcal{P}$ denote the set
\begin{equation*}
	\mathcal{P}:=\left\{(w(1),\ldots,w(k))\, : \,w(i) \geq 0,\sum_{i=1}^k w(i)f_if_i^*=I_n\right\},
\end{equation*}
where each $f_if_i^*$ is an $n\times n$ Hermitian matrix called the \emph{outer product} of $f_i$ with itself. 

Using basic polytope theory \cite{ziegler} authors in \cite{noteOnScalableFrames} described the structure of $\mathcal{P}$:

\begin{theorem}[\cite{noteOnScalableFrames}]
Suppose $\F=\{f_i\}_{i\in I}\sub H_n$ is a unit norm frame and $\ms$ is the set of its minimal scalings. Then $\mathcal{P}=\conv\ms$ and $w\in \mathcal{P}$ if and only if $w$ is a scaling of $\F$.
\end{theorem}
In other words, $\mathcal{P}$ is the convex hull of the minimal scalings, and every scaling is a convex combination of minimal scalings.

Let us begin by first defining the scalability poset $\mathbb{S}$.

\begin{definition}
Let $\F=\{f_i\}_{i\in I}$ be a frame in $H_n$. We define its \emph{scalability poset} to be the set $\mathbb{S}_F\subseteq 2^I$ ordered by set inclusion, where
\begin{equation*}
	\mathbb{S}_F=\{J\subseteq I\, :\,\{f_j\}_{j\in J}\textrm{ is a scalable frame}\}.
\end{equation*}
We assume $\emptyset\in\mathbb{S}_F$.
\end{definition}

Given a vector $f\in H_n$, $\supp(f)$ is the set of indices where the vector $f$ has nonzero components. We observe that the empty cover of the scalability poset corresponds precisely to the support of the minimal scalings of a frame $\F$, i.e. $EC(\mathbb{S})=\{\supp (v_i)\}_{i=1}^m$. This result follows from the following proposition:

\begin{proposition}
	Let $\F$ be a frame in $\R^n$ and $\ms$ its set of minimal scalings. Then $\mathrm{supp}(v_i)\neq\mathrm{supp}(v_j)$ for all $i\neq j$.
\end{proposition}

	\begin{proof}
	
	Let $u$ and $v$ be two minimal scalings with equal supports. From \cite[Prop. 3.6]{tfs} we know that $u$ and $v$ must belong to the kernel of $\tG$, the Gramian of the diagram vectors $\{\diag{f}_i\}_{i\in I}$. Consider the function $M(t) = u-tv\in \ker \tG$. Let $t_0>0$ be the smallest $t>0$ so that $\supp\{M(t)\}\subsetneq \supp(u)=\supp(v)$. Then by the definition of minimal scaling and since $M(t)$ is a scaling to a tight frame (which is not necessarily Parseval) for any $0\leq t\leq t_0$, it must be that $\supp(M(t_0))=\emptyset$. Hence, $u = t_0v$. Because both $u$ and $v$ are scalings which induce Parseval frames, this forces $t_0=1$ by a simple norm argument. Therefore, $u=v$, completing the proof.
\end{proof}

We now present results on when frames can be scaled to be prime. We prove these using the empty cover which contains useful information about when prime and non-prime scalings are possible. A scaling of a frame is \emph{prime} if the scaled frame does not contain any proper, tight subframes and \emph{non-prime} otherwise. The following theorem characterizes non-prime scalings.

\begin{theorem}\label{thm1}
A scaling is non-prime if and only if it is a convex combination of minimal scalings which can be partitioned into two orthogonal subsets.
\end{theorem}

	\begin{proof}
	We first prove the forward direction. Let $w\in\mathbb{R}^n$ be a non-prime scaling of $\F=\{f_i\}_{i\in I}$. Since $\{\sqrt{w(i)}f_i\}_{i\in I}$ is a divisible frame, there exists $J\subseteq \{1,\ldots,k\}$ such that $\{\sqrt{w(i)}f_i\}_{i\in J}$ is a tight frame. Let $K=\{1,\ldots,k\}\setminus J$. Hence $\{\sqrt{ w(i)}f_i\}_{i\in K}$ is also a tight frame. Let $w_1$ and $w_2$ denote the scalings of $\F$ for the tight subframes $\{\sqrt{w(i)}f_i\}_{i\in J}$ and $\{\sqrt{w(i)}f_i\}_{i\in K}$, where some frame vectors have coefficients equal to 0. 

Observe that the subframes $\{\sqrt{w(i)}f_i\}_{i\in J}$ and $\{\sqrt{w(i)}f_i\}_{i\in K}$ are not Parseval. However, there exists $\lambda\in (0,1)$ such that $\{\sqrt{\frac{w(i)}{\lambda}}f_i\}_{i\in J}$ and $\{\sqrt{\frac{w(i)}{1-\lambda}}f_i\}_{i\in K}$ are Parseval frames and $\frac{1}{\lambda}w_1,\frac{1}{1-\lambda}w_2\in\mathcal{P}$. Thus we may write $\frac{1}{\lambda}w_1$ and $\frac{1}{1-\lambda}w_2$ as convex combinations of the minimal scalings $\{v_i\}_{i=1}^m$, i.e.
	
		\begin{equation*}
		\frac{1}{\lambda}w_1 = \sum_{i=1}^m \alpha_i v_i, \quad	\frac{1}{1-\lambda}w_2 =\sum_{i=1}^m\beta_i v_i,
		\end{equation*}
where $\alpha_i,\beta_i\geq 0$ for all $i$ and $\sum_{i=1}^m \alpha_i=\sum_{i=1}^m\beta_i=1$. Since $J\cap K=\emptyset$, $\mathrm{supp}\left(\frac{1}{\lambda}w_1\right)$ and $\mathrm{supp}\left(\frac{1}{1-\lambda}w_2\right)$ are disjoint, and hence $\frac{1}{\lambda}w_1$ and $\frac{1}{1-\lambda}w_2$ are orthogonal. This implies that
		\begin{align*}
		0&=\left\langle \frac{1}{\lambda}w_1,\frac{1}{1-\lambda}w_2\right\rangle\\
		&=\left\langle \sum_{i=1}^m \alpha_iv_i,\sum_{i=1}^m\beta_i v_i\right\rangle \\
		&=\underbrace{\sum_{i=1}^m\langle \alpha_iv_i,\beta_iv_i\rangle}_{A} + \underbrace{\sum_{i\neq j} \langle \alpha_iv_i,\beta_jv_j\rangle}_B .
		\end{align*}
As $\alpha_i,\beta_i\geq 0$ for all $i$ and each minimal scaling $v_i$ lies in $\mathbb{R}_{\geq 0}^k$, the terms $A$ and $B$ must both sum to 0. Considering each term separately, $A=0$ implies that whenever $\alpha_i>0$, $\beta_i=0$ and whenever $\beta_i>0$, $\alpha_i =0$. Thus, the minimal scalings which appear nontrivially in the expression $\sum_{i=1}^m \alpha_i v_i$ appear trivially in the expression $\sum_{i=1}^m\beta_i v_i$ and vice versa. Moreover, since $B=0$, the minimal scalings which appear nontrivially in $\sum_{i=1}^m \alpha_i v_i$ must be orthogonal to those which appear nontrivially in $\sum_{i=1}^m\beta_i v_i$. Letting
	\begin{equation*}
	C=\{v_j\, :\,\alpha_j>0\}, \quad D=\{v_\ell\, :\,\beta_\ell>0\},
	\end{equation*}
we see that $C$ and $D$ are orthogonal subsets of $\{v_i\}_{i=1}^m$. Since
	\begin{align*}
	w&=w_1+w_2\\
	&=\sum_{j\in C}\lambda\alpha_jv_j+\sum_{\ell\in D}(1-\lambda)\beta_\ell v_\ell\\
	&=\sum_{i\in C\cup D}\left[\lambda\alpha_i+(1-\lambda)\beta_i\right]v_i,
	\end{align*}
$w$ is a convex combination of minimal scalings which can be partitioned into two orthogonal subsets, which proves the forward direction.

For the reverse direction, suppose $w$ is a scaling which can be expressed as a convex combination of minimal scalings partitioned into two orthogonal subsets. Thus there exists  $C,D\subseteq \{1,\ldots,m\}$ such that $C \cap D = \emptyset$ and $\langle v_j,v_l\rangle =0$ for all $j\in C, \ell\in D$, and we can write
	\begin{equation*}
	w=\underbrace{\sum_{j\in C}\alpha_j v_j}_{w_1}+\underbrace{\sum_{\ell\in D}\beta_\ell v_\ell}_{w_2},
	\end{equation*}
where $\alpha_j,\beta_\ell \geq 0$ for all $j\in C$ and $\ell\in D$, and $\sum_{j\in C}\alpha_j +\sum_{\ell\in D}\beta_\ell=1$. Note that $w_1$ and $w_2$ are both tight scalings of $\F$ and $\mathrm{supp}(w_1)\cap\mathrm{supp}(w_2)=\emptyset$. This implies that
	\begin{equation*}
	\{\sqrt{w_1(i)}f_i\}_{i\in I},\{\sqrt{w_2(i)}f_i\}_{i\in I}\subseteq\{\sqrt{w(i)}f_i\}_{i\in I}
	\end{equation*}
are both tight frames, and therefore $w$ is a non-prime scaling. This completes the proof. 
	\end{proof}

\begin{example}
Let $\F=\{e_1,e_2,-e_1,-e_2\}\subseteq \mathbb{R}^2$. Its minimal scalings are $v_1=(1,1,0,0)$, $v_2=(0,0,1,1)$, $v_3=(1,0,0,1)$, $v_4=(0,1,1,0)$ and note that $v_1\perp v_2$ and $v_3\perp v_4$. Let $\alpha_1,\alpha_2\in (0,1)$, $\alpha_1+\alpha_2=1$. By Theorem \ref{thm1}, $w_1=\alpha_1v_1+\alpha_2v_2$ and $w_2=\alpha_1v_3+\alpha_2v_4$ are non-prime scalings. Furthermore, any non-prime scaling of $\F$ can be expressed as $w_1$ or $w_2$. 

\end{example}

\begin{definition}
A scaling is called \emph{strict} if all of its entries are strictly positive.  
\end{definition}

\begin{theorem}\label{thm:noPrimeStrict}
A unit-norm frame $F = \fI$ has no prime strict scalings if and only if the empty cover of its scalability poset can be partitioned into two disjoint sets $A$ and $B$ so that 
%every index $i \in I$ appears in at least one of $A$ and $B$ but never both.  
$\bigcup_{K\in A}K$ and $\bigcup_{K\in B}K$ are disjoint sets of indices.  
\end{theorem}

Before proving this theorem, we present the following lemma regarding covering convex sets by vector spaces.  Its proof, which reduces the claim to proving that any finite number of points in a convex set $C$ must be contained in a single vector space $V$, is elementary and is therefore omitted.

\begin{lemma}
If a convex set $C$ in $\R^n$ is contained in $\bigcup_{i \in I} V_i$, where each $V_i$ is a subspace of $\R^n$, with $|I|$ finite, then there is at least one index $i$ such that $C \subseteq V_i$.  
\end{lemma}

To proceed with the proof of the theorem, we first make an observation. Let $\Poly$ be the polytope consisting of all scalings of a given frame in $\R^n$. Using the interpretation of scalings in \cite{tfs}, this polytope is the intersection of the vector space $\ker(\tG)$, the null space of the Gramian of the diagram vectors for the frame, with the convex region $R = \R_{\geq 0}^k\cap \{v \in \R^k:\|v\|_1=n\} = \{v\in \R^k: \sum_{i=1}^k v_i = n,\; v_i\geq 0 \;\forall\, i\}$. Using this we prove our result on prime strict scalings.

\begin{proof}[Proof of Theorem \ref{thm:noPrimeStrict}]
By Theorem \ref{thm1} we know that a scaling is non-prime if and only if it can be written as a non-trivial convex combination of two orthogonal scalings.  Consider any subset $J\sub I=\{1,\ldots,k\}$ and its complement $J^c$.  Some scalings will be supported on $J$, and some scalings will be supported on $J^c$, and some scalings are not supported on either.  The set of all vectors (not scalings) supported on $J$ forms a vector space $A = \Span\{e_j\;:\; j \in J\}$, and the set of all scalings supported on $J$ will be the set $A \cap \ker(\tG) \cap R$.  Similarly, the set of all scalings supported on $J^c$ is $B \cap \ker(\tG) \cap R$ where $B=\Span\{e_j\;:\; j \in J^c\}$.  Then any scaling that is a nontrivial non-negative convex combination of a vector in $A \cap \ker(\tG) \cap R$ and a vector in $B \cap \ker(\tG) \cap R$, and is itself a scaling, will be a non-prime scaling by Theorem \ref{thm1}.  Because $A$ and $B$ are orthogonal subspaces with disjoint supports, this set corresponds exactly to $((A \cap \ker(\tG)) \oplus (B \cap \ker(\tG))) \cap \inter(R)$.  Note that the relative interior of $R$, $\inter(R)$, is those points in $R$ with all components strictly positive.  We restrict to the interior to avoid trivial combinations, which are not prime, as well as any non-strict scalings.  Observe that this region is of the form $C_J \cap \inter(R)$ where $C_J$ is a vector space dependent upon $J$.

Now observe that every non-prime strict scaling will be in $C_J \cap \inter(R)$ for some $J$, because any two orthogonal scalings will fall into some partition $(J|J^c)$, and as noted above, every scaling in $C_J \cap \inter(R)$ will be non-prime.  Hence the property that a frame has no prime and strict scalings is equivalent to $\inter(\Poly)\sub\cup C_J$.  Yet $\inter(\Poly)$ is convex, and therefore the preceding lemma applies. It follows that there exists some $C_J$ containing all of $\inter(\Poly)$.  Because $C_J$ (a finite dimensional subspace) is closed, $C_J$ contains all of $\Poly$.  

Therefore, every point in $\Poly$ is a linear combination of some scaling supported on $J$ and one supported on $J^c$; in fact, it is a convex combination, because every point in $\Poly$ has only positive coefficients, and the two scalings have disjoint support.  Because minimal scalings (corresponding to the empty cover elements) cannot be written as convex combinations of one another, it must be the case that every element of the empty cover is either a subset of $J$ or a subset of $J^c$.  This proves the forward direction of the theorem.  

For the converse, assume that the empty cover of the scalability poset can be partitioned into two sets $A$ and $B$ so that $\bigcup_{K\in A}K$ and $\bigcup_{K\in B}K$ are disjoint sets of indices.  Then every strict scaling can be written as a non-negative linear combination $\sum_{a \in A} \lambda_a s(a) + \sum_{b \in B} \lambda_b s(b)$ where $s(a)$ is the minimal scaling associated to the empty cover element $a\in EC(\mathbb{S}_F)$.  But since $A$ and $B$ have disjoint supports, this is a sum of two orthogonal scalings, and it must be non-trivial if the scaling is strict.  Hence every strict scaling must be non-prime.  This completes the proof.  
\end{proof}

\begin{proposition}\label{usupp}
Let $A\in\mathbb{S}_F$ and $J=\{j\, :\,\mathrm{supp}(v_j)\subseteq A\}$. Suppose that $\{v_j\}_{j\in J}$ cannot be partitioned into orthogonal subsets. If $\mathrm{supp}(v_j)\not\subseteq\bigcup\limits_{\ell\in J\setminus\{j\}}\mathrm{supp}(v_\ell)$ for all $j$, then every scaling $w$ with $\mathrm{supp}(w)=A$ is prime.
\end{proposition}

	\begin{proof}
	Suppose toward a contradiction that $w$ is a non-prime scaling with $\mathrm{supp}(w)=A$. By Theorem \ref{thm1}, there exists $L\subseteq J$ such that $w$ is a convex combination of minimal scalings $\{v_\ell\}_{\ell\in L}$, where $\{v_\ell\}_{\ell\in L}$ can be partitioned into orthogonal subsets and $\bigcup_{\ell\in L}\mathrm{supp}(v_\ell)=A$. Since $\{v_\ell\}_{\ell\in L}$ can be partitioned but $\{v_j\}_{j\in J}$ cannot be partitioned, $L\subsetneq J$. Let $p\in J\setminus L$. Then $\mathrm{supp}(v_p)\subseteq \bigcup_{\ell\in J\setminus\{p\}}\mathrm{supp}(v_\ell)$, a contradiction. Hence every scaling $w$ with $\mathrm{supp}(w)=A$ is a prime scaling.
	\end{proof}

The following example shows that if $\{v_j\}_{j\in J}$ cannot be partitioned into orthogonal subsets and there exists $j$ such that  $\mathrm{supp}(v_j)\subseteq\bigcup\limits_{\ell\in J\setminus\{j\}}\mathrm{supp}(v_\ell)$, both prime and non-prime scalings are possible.

\begin{example}
Let $\F=\left\{\colvec{2}{1}{0}, \colvec{2}{-\frac{1}{2}}{\frac{\sqrt{3}}{2}}, \colvec{2}{-\frac{1}{2}}{-\frac{\sqrt{3}}{2}}, \colvec{2}{0}{1},\colvec{2}{-\frac{\sqrt{3}}{2}}{-\frac{1}{2}},\colvec{2}{\frac{\sqrt{3}}{2}}{-\frac{1}{2}}\right\}\sub\R^2$. The minimal scalings of $\F$ are $v_1=\left(\frac{2}{3},\frac{2}{3},\frac{2}{3},0,0,0\right), v_2=\left(0,0,0,\frac{2}{3},\frac{2}{3},\frac{2}{3}\right), v_3=(1,0,0,1,0,0), v_4=(0,1,0,0,1,0), v_5=(0,0,1,0,0,1)$. Note that the minimal scalings cannot be partitioned into orthogonal subsets and $\mathrm{supp}(v_i)\subseteq\bigcup_{j\in I\setminus\{i\}}\mathrm{supp}(v_j)$ for $i=1,\ldots,5$. The scaling $w_1=\sum_{i=1}^5\alpha_iv_i$, where $\alpha_i\neq\alpha_j$ for $i\neq j$ and $\alpha_i>0$ for all $i$ is prime. On the other hand, the scaling $w_2=\sum_{i=1}^5 \frac{1}{5}v_i$ is non-prime, although it is a convex combination of minimal scalings which cannot be partitioned into orthogonal subsets. Note that $w_2$ can be written as a convex combination of the minimal scalings in many ways, e.g.
	\begin{align*}
	w_2&=\frac{1}{2}(v_1+v_2)\\
	&=\frac{1}{3}(v_3+v_4+v_5)\\
	&=\frac{1}{4}(v_1+v_2)+\frac{1}{6}(v_3+v_4+v_5),
	\end{align*}
Since there exists at least one way of expressing $w_2$ as a convex combination of minimal scalings which can be partitioned into orthogonal subsets, $w_2$ is a non-prime scaling  by Theorem \ref{thm1}. 
\end{example}

We conclude this paper with an open problem $(Q)$: given a poset $P$ with $I\in P$ and a unit-norm frame $\F=\fI\subseteq\mathbb{R}^n$, does there exist a strictly scaled version $\widehat{\F}$ of $\F$ so that the factor poset of $\widehat{\F}$ is $P$? In this paper we have considered a particular case of $(Q)$, namely, when $P$ corresponds to a prime configuration: $P=\{\emptyset,I\}$ (Theorem \ref{thm:noPrimeStrict}). In general note that an obvious necessary condition for $(Q)$ to hold is that $P\subseteq\mathbb{S}_F$. While we have not resolved $(Q)$, we believe that it can be answered based on the structure of scalability posets alone. Along these lines, we close with a conjecture.

\begin{conjecture}
If the unit-norm frames $F =\fI,G=\gI\sub\R^n$ have the same scalability poset and $G$ is a tight frame, then $F$ can be strictly scaled to have the same factor poset as $G$.
\end{conjecture}

\section*{Acknowledgements}

Much of this work was done when A. Chan, L. Stokols, and A. Theobold participated in the Central Michigan University NSF-REU program in the summer of 2013. M.S. Copenhaver and S.K. Narayan were Graduate Student Mentor and Faculty Mentor, respectively.

%%%%%%%%%%%%% References

\bibliographystyle{amsplain}
\bibliography{References}

\end{document}